\documentclass[12pt,reqno]{amsart}
   
\usepackage[left=1in, right=1in, top=1.1in,bottom=1.1in]{geometry}
\usepackage[dvipsnames]{xcolor}
\usepackage[mathscr]{eucal}
\usepackage{latexsym,amsfonts,amssymb}
\usepackage{bm,pifont,stmaryrd}
\usepackage{graphicx}
\usepackage{color}
\usepackage{enumitem}
\usepackage{hyperref}
\hypersetup{
     colorlinks   = true,
     citecolor    = blue,
     linkcolor    = blue
} 
\usepackage{showlabels}
  
\usepackage{pdfsync}
\usepackage[font={scriptsize}]{caption}

\usepackage{tikz}
\usetikzlibrary{trees}

\newcommand{\ep}{\varepsilon}

\newcommand{\DD}{\mathbb{D}}

\newtheorem{theorem}{Theorem}[section]
\newtheorem{Def}[theorem]{Definition}
\newtheorem{notation}[theorem]{Notation}
\newtheorem{thm}[theorem]{Theorem}
\newtheorem{proposition}[theorem]{Proposition}
\newtheorem{prop}[theorem]{Proposition}
\newtheorem{cor}[theorem]{Corollary}
\newtheorem{lemma}[theorem]{Lemma}
\newtheorem{example}[theorem]{Example}
\theoremstyle{remark}
\newtheorem{remark}[theorem]{Remark}
\newtheorem{hyp}[theorem]{Hypothesis}
\numberwithin{equation}{section}

\newcommand{\id}{\text{Id}}

\def\R{\mathbb{R}}
\def\RR{\mathbb{R}}

\def\NN{\mathbb{N}}
\def\mE{\mathbb{E}}
\def\EE{\mathbb{E}}

\def\bfw{{\bf w}}
\def\bfx{{\bf x}}

\newcommand{\ca}{{\mathcal A}}
\newcommand{\cb}{{\mathcal B}}

\newcommand{\ce}{{\mathcal E}}
\newcommand{\cf}{{\mathcal F}}
\newcommand{\cg}{{\mathcal G}}
\newcommand{\ch}{{\mathcal H}}

\newcommand{\ck}{{\mathcal K}}
\newcommand{\cl}{{\mathcal L}}
\newcommand{\cm}{{\mathcal M}}

\newcommand{\cs}{{\mathcal S}}

\newcommand{\cv}{{\mathcal V}}

\def\si{\sigma}

\def\al{{\alpha}}

\def\be{{\beta}}

\def\ga{{\gamma}}

\newcommand{\lcl}{\left\{}
\newcommand{\rcl}{\right\}}
\newcommand{\lp}{\left(}
\newcommand{\rp}{\right)}
\newcommand{\lc}{\left[}
\newcommand{\rc}{\right]}

\def \eref#1{\hbox{(\ref{#1})}}

\def\tvr{\text{-var}}

\def\ll{\llbracket}
\def\rr{\rrbracket}

\def \eref#1{\hbox{(\ref{#1})}}

\begin{document}
\title[Uniform bounds for Euler scheme]
{Euler scheme for SDEs driven 
 by   fractional Brownian motions: 
  Malliavin differentiability and uniform upper-bound estimates}
\date{}   

\author[J. A. Le\'on]{Jorge A. Le\'on}
\address{J. A. Le\'on: Departamento de Control Autom\'atico, Cinvestav-IPN, Mexico}
\email{jleon@ctrl.cinvestav.mx}

\author[Y. Liu]{Yanghui Liu}
\address{Y. Liu: Baruch College, CUNY, New York}
\email{yanghui.liu@baruch.cuny.edu}

\author[S. Tindel]{Samy Tindel}
\address{S. Tindel: Department of Mathematics, 
Purdue University, West Lafayette}
\email{stindel@purdue.edu}

\begin{abstract} 
The Malliavin differentiability of a  SDE plays a   crucial role  in the study of density smoothness and ergodicity  among others. For Gaussian driven SDEs the differentiability  issue is   solved essentially in \cite{CLL}. 
  In this paper, we consider the Malliavin differentiability   for the Euler scheme of such SDEs. We will focus  on SDEs driven by   fractional Brownian motions (fBm), which is a very natural class of Gaussian processes. 
 We derive a uniform (in    step size $n$)  path-wise upper-bound estimate for    the Euler scheme for stochastic differential equations driven by fBm with Hurst parameter $H>1/3$ and its Malliavin derivatives.    
\end{abstract}


\maketitle

{
}

\section{Introduction}

 In this paper we are interested in    the following stochastic  differential equation driven by a $d$-dimensional fractional Brownian motion (fBm in the sequel) $x$ with Hurst parameter $ \frac13 <   H < \frac12$:
 \begin{eqnarray}\label{e1.1}
 dy_{t}&=&V_{0}(y_{t})dt + {V}(y_{t}) d x_{t}\,, \quad t\in [0,T],
\\
y_{0}&=&a \in \RR^{m}.
\nonumber
\end{eqnarray}
Throughout the paper we assume that the collection of vector fields $V_{0} = (V_{0}^{i}, {1\leq i \leq m})  \in C^{3}_{b} (\RR^{m}, \RR^{m})  $ and ${V} = (V^{i}_{j}, {1\leq i\leq m, 1\leq j\leq d})$ all sit in the class $C^{3}_{b} (\RR^{m}, \cl (\RR^{d}, \RR^{m})) $.
 The existence and uniqueness of   path-wise solution of equation \eqref{e1.1} 
is guaranteed by   the theory of rough paths; see e.g. \cite{FV}. In addition, the unique solution $y$ in the   sense of \cite{FV} has $\ga$-H\"older continuity for all $\ga<H$. 

The aim of this paper is to consider the numerical approximation of equation \eqref{e1.1}. 
It is well-known (see the introduction in~\cite{DNT} for more details about this issue) that the classical Euler scheme is divergent under this setting.  The simplest possible solution to this problem is to use a second-order Euler (that is a Milstein type) scheme,   which  however      involves   iterated integrals of the fBm $x$ and  is not implementable directly.  
Several contributions are made to   tackle the implementation issue  \cite{DNT, FR, HLN1, LT}; see also \cite{HLN,HLN2}. 

In this paper we will focus our attention on the  (implementable) Euler scheme introduced in~\cite{HLN1,LT}. 
Take  the uniform partition $\pi: 0=t_{0}<t_{1}<\cdots<t_{n}=T$ on $[0,T]$, where  for $k=0,\ldots,n$ we have $t_{k} =  k \Delta$  with $\Delta = \frac{T}{n}$. 
The Euler scheme 
is recursively defined as follows:
\begin{equation}\label{e4}
y^{n}_{t_{k+1}} = y^{n}_{t_{k}} + V_{0}(y^{n}_{t_{k}}) \Delta+ V(y^{n}_{t_{k}}) \delta x_{t_{k}t_{k+1}} 
+\frac12 \sum_{ j=1}^{d}\partial V_{j} V_{j} (y^{n}_{t_{k}}) \, \Delta^{2H} 
\end{equation}
and
$y^{n}_{0}=y_{0}$, 
where we have used the notation   
\begin{eqnarray}\label{not:iterated-vector-field} 
\partial V_{i} V_{j} =
\lp
 \sum_{\ell=1}^{m} \partial_{\ell} V_{i}^{k} V_{j}^{\ell}; \, k=1,\dots, m
 \rp 
\end{eqnarray}
 and $\partial_{i}$ stands for the partial derivative in the $y^{(i)} $ direction: $\partial_{i} = \frac{\partial }{\partial {y^{(i)} } }$. 
 The exact rate of convergence of $y^{n}$ to $y$ is shown to be of order $1/n^{2H-1/2}$ in~\cite{LT}.

In this paper, we are   interested in proving that the approximation $y^{n}$ defined by \eqref{e4} is Malliavin differentiable under sufficient smoothness assumption on  the coefficients. More importantly, we will establish pathwise upper bounds estimates of the Malliavin derivative which will be uniform in $n$. Our motivation for this endeavor is twofold:

\begin{enumerate}[wide, labelwidth=!, labelindent=0pt, label=\textbf{(\roman*)}]
\setlength\itemsep{.1in}
\item 
The integrability of Malliavin derivatives for rough differential equations has been an important open problem a decade ago. This is mostly  due to the prominent role played by Malliavin calculus techniques in obtaining results about the density of random variables like $y_{t}$ in~\eqref{e1.1}. The integrability issue for the Malliavin derivatives $Dy_{t}$ has been solved 
completely in~\cite{CLL}. Subsequent applications to the smoothness of densities of $y_{t}$ are contained in~\cite{bau,CHLT, geng}. The corresponding question for numerical approximations of $y$ is thus in order. We propose to start a detailed answer to this natural problem in the current paper.

\item
Upper bounds on Malliavin derivatives open the way to important results for numerical schemes. Among others, one can quote weak convergence as well as convergence of densities. In our companion paper~\cite{LLT1} we prove the weak convergence of $y^{n}$ defined by~\eqref{e4} towards the solution to~\eqref{e1.1}. The uniform bounds on Malliavin derivatives obtained in the current contribution are a crucial ingredient in~\cite{LLT1}. 

\end{enumerate}

With those motivations in  mind, our main result can be informally spelled out as follows. Please refer  to  Theorem \ref{thm.bd}, Remark~\ref{remark.hnorm} and Theorem \ref{thm.xirr} for a more  precise statement. 
\begin{thm}\label{thm:main-intro}
Let $y$ and $y^{n}$ be the solution of \eqref{e1.1} and the corresponding Euler scheme~\eref{e4}, respectively.  Take an integer  $L\geq 1$. Let $\bar{D}^{L}y^{n}_{t}$ be the $L$th Malliavin derivative of $y^{n}_{t}$ in the Cameron-Martin space $\bar{\ch}$ corresponding to the fBm $x$. 
 Suppose   that $V\in C^{L+2}_{b}$.  
 Then for each    $n\in\NN$ there is a functional  $\cg_{L}^{n}$ of the fBm $x$ which is almost surely finite and such that the following pathwise bound holds true:
\begin{equation}\label{a1}
\|\bar{D}^{L}y^{n}_{t}\|_{\bar{\ch}^{\otimes L}} 
 \leq \cg_{L}^{n} ,
 \quad\text{for all }
 t=t_{k} \text{ and } k=1,\dots,n.
\end{equation}
 The explicit expression of $\cg_{L}$ is given in  Theorem \ref{thm.bd}. Furthermore,  we have the uniform integrability of $\cg^{n}_{L}$ for   $n\in\NN$: 
   \begin{eqnarray*}
\sup_{n\in\NN}\mE[|\cg^{n}_{L}|^{p}]<\infty . 
\end{eqnarray*}
 \end{thm}

As mentioned above, Theorem~\ref{thm:main-intro} is a crucial step in the analysis of weak convergence for the Euler type scheme~\eqref{e4}. In addition, the proof of our main estimate~\eqref{a1} relies on techniques which are interesting in their own right. Specifically, we will first resort to rough paths type estimates (recalled in Section~\ref{section.pre}), and our Malliavin calculus setting will follow Inahama's approach~\cite{Inahama} for all computations in the Cameron-Martin space. On top of those classical ingredients, our main technical tool will be a representation of higher order Malliavin derivatives of $y^{n}$ in terms of a tree expansion (see Lemma~\ref{lem.dfg} and Lemma~\ref{lem.dhd} below). This kind of expression has to be contrasted with the standard form of higher order Malliavin derivatives, based on sums over partitions of the set $\{1,\ldots,n\}$ (see \cite[Proposition 5]{NS}).   We note that the advantage of   
  our directed-tree notations is that it allows us to  distinguish all terms  in the  chain  differentiations. We will  benefit from this feature while proving an identities~\eqref{eqn.dxin}-\eqref{eqn.dly} in Lemma \ref{lem.dhd}.  We also believe that the tree-based computations presented here can be usefully applied to other numerical schemes. We plan on developping this line of research in subsequent publications.
 
  The paper is organized as follows.  In Section \ref{section.pre} we recall  some basic material on rough paths and Malliavin calculus. We also  review   results on the Euler scheme  which  will be used throughout the paper.  
    In Section \ref{section.der} we derive a representation of Malliavin derivatives of the Euler scheme via tree notations. Finally, in Section \ref{section.bd} we prove the uniform upper-bound estimate for the Euler scheme and its Malliavin derivatives.

\begin{notation}\label{general-notation}
In what follows,  we take $n\in \NN$ and   $\Delta =T/n$,  and  consider the uniform partition:   
  $0=t_{0}<t_{1}<\cdots<t_{n}=T$ on $[0, T]$, where  $t_{k} =k\Delta$. We denote by   $\ll s, t\rr$ the discrete interval: $\ll s, t\rr = \{t_{k}\in [s,t]: k=0,\dots, n\}$. For $u\in [t_{k}, t_{k+1})$, we denote $\eta(u) =t_{k}$. 
 For an interval $[s,t]\subset [0,T]$  we define the simplex $\cs_{2}([s,t]) = \{(u,v): s\leq u\leq v \leq t\}$.  
For a vector $a=(a^{1}, \dots, a^{m})\in \RR^{d}$ we define the norm $|a|=\max_{j=1,\dots,m}|a_{j}|$. 
Throughout the paper, we use $C$ and $K$ to represent constants that are independent of $n$ and whose values may change from line to line.
\end{notation}
 
\section{Preliminary results} \label{section.pre}

In this section we recall some basic notions of rough paths theory and their application to fractional Brownian motion, which allows a proper definition of equation~\eqref{e1.1}. We also give the necessary elements of Malliavin calculus in order to estimate densities of random variables.
  
\subsection{Elements of rough paths and fractional Brownian motion}\label{section2}

This subsection is devoted to introduce some basic concepts of rough paths theory. We are going to restrict our analysis to a generic H\"older regularity of the driving path of order $\frac13<\ga\leq \frac12$, in order to keep expansions to a reasonable size. We also fix a finite time horizon $T>0$. The following notation will prevail until the end of the paper: for a Banach space $\cv$ (which can be either finite or infinite dimensional) and two functions $f\in C([0,T],\cv)$ and $g\in C(\cs_{2}([0,T]),\cv)$ we set 
\begin{eqnarray}\label{eq:def-delta}
\delta f_{st} = f_{t}-f_{s},
\quad\text{and}\quad
\delta g_{sut} = g_{st}-g_{su}-g_{ut}, 
\qquad 0\leq s\leq u\leq t\leq T.
\end{eqnarray}

Let us introduce the analytic requirements in terms of H\"older regularity which will be used in the sequel. Namely consider two paths $x \in C([0,T], \RR^{d})$ and $x^{2} \in C(\cs_{2}([0,T]), (\RR^{d})^{\otimes 2})$. Then we denote
\begin{equation}\label{eq:def-holder-seminorms}
\|x\|_{[s,t], \ga}:=\sup_{(u,v)\in\cs_{2}([s,t]): u\neq v}\frac{|\delta x_{uv}|}{|v-u|^{\ga}} , 
\qquad  
\|x^{2}\|_{[s,t],  2\ga}:= \sup_{(u,v)\in\cs_{2}([s,t]): u\neq v}\frac{|  x^{2}_{uv}| }{|v-u|^{ 2\ga}} . 
\end{equation}
 When the semi-norms in \eqref{eq:def-holder-seminorms} are finite we say that $x$ and $x^{2}$ are respectively in $C^{\ga}([s,t], \RR^{d})$ and $C^{2\ga}(\cs_{2}([s,t]), (\RR^{d})^{\otimes 2})$.
For convenience, we denote $ \|x\|_{\ga}:= \|x\|_{[0,T], \ga} $ and $ \|x^{2} \|_{2\ga }:= \|x^{2}\|_{[0,T], 2\ga } $.
With this preliminary notation in hand, we can now turn to the definition of rough path.

\begin{Def}\label{def:rough-path}
Let $x \in C([0,T], \RR^{d})$, $x^{2} \in C(\cs_{2}([0,T]), (\RR^{d})^{\otimes 2})$, and $\frac13<\ga\leq \frac12$. For $(s,t)\in \cs_{2}([0,T])$ we denote $x^{1}_{st}=\delta x_{st}$. 
We call $ \bfx:=S_{2}(x):=(x^{1}, x^{2})  $ a (second-order) $\ga$-rough path if $ \|x^{1}\|_{\ga} <\infty $ and $\|x^{2}\|_{2\ga}<\infty$, and if the following algebraic relation holds true: 
\begin{eqnarray}\label{eqn.delta.def}
\delta x^{2}_{sut} =x^{2}_{st} - x^{2}_{su} - x^{2}_{ut} = x_{su}^{1}\otimes x_{ut}^{1}
\qquad s\leq u\leq t ,
\end{eqnarray}
where we have invoked \eqref{eq:def-delta} for the definition of $\delta x^{2}$. For a $\ga$-rough path $S_{2}(x)$, we define a $\ga$-H\"older semi-norm as follows:
\begin{eqnarray}\label{eq:def-norm-rp}
\|S_{2}(x)\|_{\ga} := \|x^{1}\|_{\ga}+ \|x^{2}\|_{2\ga}^{\frac12}\,. 
\end{eqnarray}
An important subclass of rough paths are the so-called \emph{geometric $\ga$-H\"older rough paths}.  A geometric $\ga$-H\"older rough path is a  $\ga$-rough path $  (x, x^{2})$  such that  there exists a sequence of smooth $\RR^{d}$-valued paths $(x^{n}, x^{2,n})$ verifying:
\begin{eqnarray}\label{eq:cvgce-for-geom-rp}
\| x-x^{n}\|_{\ga} + \|x^{2}-x^{2,n}\|_{2\ga}  \rightarrow 0,  \quad \text{ as } n \rightarrow \infty . 
\end{eqnarray}
We will mainly consider geometric rough paths in the remainder of the article.
\end{Def}

Let  $x$ be a rough path as given in Definition~\ref{def:rough-path}. We shall interpret equation \eqref{e1.1} in a way introduced by Davie in \cite{D}, which is conveniently compatible with numerical approximations.

 \begin{Def}\label{def:diff-eq-davie}
Let $(x, x^{2})$ be a  $\ga$-rough path with $\ga>1/3$. We say that $y$ is a solution of~\eref{e1.1} on $[0,T]$ if $y_{0} = a$ and there exists a constant $K>0$ and $\mu>1$  such that 
\begin{equation}\label{eq:dcp-Davie}
\Big| \delta y_{st} -  \int_{s}^{t} V_{0}(y_{u}) \, du - V(y_{s})  x^{1}_{st} 
- \sum_{i,j=1}^{d} \partial V_{i}V_{j} (y_{s} ) x^{2,ij}_{st} \Big| 
\leq 
K |t-s|^{\mu}
\end{equation}
for all $(s,t) \in \mathcal{S}_{2}([0,T])$, where we recall that $\delta y$ is defined by~\eref{eq:def-delta} and the notation $\partial V_{i}V_{j}$ is introduced in \eqref{not:iterated-vector-field}. 
\end{Def}
\noindent
According to \cite{D} there exists a unique   RDE solution to equation \eref{e1.1}, understood as in Definition~\ref{def:diff-eq-davie}.

In the following we recall a sewing map lemma with respect to discrete control functions. It is an elaboration  of \cite[Lemma 2.5]{LT} and proves to be useful in the analysis of the numerical scheme.  Let  $\pi : 0=t_{0}< t_{1}<\cdots <t_{n-1}< t_{n} =T $ be a generic partition of the interval $[0,T]$ for $n \in \NN$. We denote by $ \llbracket  s, t \rrbracket  $ the discrete interval  $\{t_{k} : s\leq t_{k} \leq t \}$ for $0\leq s < t \leq T$. In this paper, a two variable function $\omega:\cs_{2}(\ll0,T\rr)\to [0,\infty)$ is called a control on $\ll0,T\rr$ if it satisfies the super-additivity condition. That is, $\omega(s,u)+\omega(u,t)\leq \omega(s,t)$ for $s,u,t\in\ll0,T\rr$ such that $s\leq u\leq t$.

 \begin{lemma}\label{lem2.4}
 Suppose that $\omega $ is a control on $\ll 0, T\rr$. 
 Consider a Banach space $\cb$ and an increment  $R : \cs_{2}(\ll 0, T \rr)   \to \cb $. 
  Suppose that $|R_{t_{k}t_{k+1}}|\leq\omega(t_{k}, t_{k+1})^{\mu}$ for all $t_{k}\in \ll0,T\rr$ and that $|\delta R_{sut}|\leq \omega (s, t)^{\mu}$ with  an  exponent $\mu>1$, where recall that $\delta R_{sut} = R_{st}-R_{su}-R_{ut}$. Then the following relation holds:
 \begin{eqnarray}\label{eqn.kmu}
 |R_{st} |  \leq K_{\mu} \omega(s, t)^{\mu} \,,
\quad\text{where}\quad
K_{\mu} = 2^{\mu} \, \sum_{l=1}^{\infty} l^{-\mu}.
\end{eqnarray}
\end{lemma}

 We now specialize our setting to a path $x=(x^1,\dots, x^d)$ defined as a standard $m$-dimensional fBm  on $[0,T]$ with Hurst parameter $H \in (\frac13, \frac12)  $. 
This fBm is defined on a complete probability space $(\Omega, \cf, \mathbb{P})$, and we assume that the $\si$-algebra $\cf$ is generated by $x$. 
 In this situation, recall that the covariance function of each coordinate of $x$ is defined on $\mathcal{S}_{2}([0,T])$ by:
\begin{eqnarray}\label{eq:cov-fbm}
R(s,t) = \frac{1}{2} \lc |s|^{2H} + |t|^{2H} - |t-s|^{2H}  \rc .
\end{eqnarray}
It  is   established in \cite[Chapter 15]{FV} that the geometric rough path $S_{2}(x)$ of $x$  via the piecewise linear approximation  is well defined for $\frac13 <\ga <H$ in the sense of  Definition \ref{def:rough-path}.

\subsection{Malliavin calculus for $\mathbf{x}$}\label{subsection.d}

 In this subsection we   recall some concepts of Malliavin calculus which will be used later in the paper. Recall that $R$ is the covariance function of the fBm $x$ defined in \eqref{eq:cov-fbm}.   
 Denote by $\mathcal{E} $ the set of step functions on the interval $  [0,T]$. We define the Hilbert space $\ch$   as  the closure of $ \mathcal{E} $ 
with respect to the scalar product
\begin{equation}\label{e.def.r}
\langle \mathbf{1}_{[u,v]}, \mathbf{1}_{[s,t]} \rangle_{\mathcal{H} } =  R ([u,v], [s,t])
\equiv 
 R(v,t) - R(v,s) - R(u,t) +R(u,s) .
\end{equation}
The space $\ch$ is very useful in order to define Wiener integrals with respect to $x$. However, in the current paper we also need to introduce the Cameron-Martin space $\bar{\ch}$ related to our driving process. The latter space is the one allowing to identify pathwise derivatives with respect to $x$ and Malliavin derivatives. In order to construct $\bar{\ch}$,
let first $\mathcal{R}$ be the linear  operator on $\ce$ such that 
\begin{eqnarray}\label{e.cr}
\mathcal{R} (\mathbf{1}_{[0,t]}) =   R(t,\cdot) \, ,
\end{eqnarray}
and we also set $\bar{\ce} = \mathcal{R}(\ce)$. Then 
we can  define the  Cameron-Martin space $\bar{\ch}$ as the closure of  $\bar{\ce}$ with respect to the inner product
\begin{eqnarray*}
\langle \mathcal{R} (\mathbf{1}_{[0,t]}), \mathcal{R} (\mathbf{1}_{[0,s]}) \rangle_{\bar{\ch}} = \langle \mathbf{1}_{[0,s]}, \mathbf{1}_{[0,t]} \rangle_{\mathcal{H} } . 
\end{eqnarray*}
It is clear that $\mathcal{R}$ is an isometry between the two Hilbert spaces $\ch$ and $\bar{\ch}$. Note that according to \eqref{e.cr} we have 
\begin{eqnarray}\label{e.rh}
\mathcal{R}(h)(t) = \langle h, \mathbf{1}_{[0,t]}\rangle_{\ch}
\end{eqnarray}
for $h\in\ce$. By the isometry property of $\mathcal{R}$ we see that \eqref{e.rh}   holds for all $h\in\ch$. We refer to \cite{GOT,NS} for more details about the spaces $\ch,\bar{\ch}$.

For the sake of conciseness, we refer to \cite{N} for a proper definition of Malliavin derivatives in the Hilbert space $\ch$ and related Sobolev spaces in Gaussian analysis. Let us just mention that for a   functional $F$ of $x$ we will denote its Malliavin derivative  by $ {D}F$, the Sobolev spaces by $\mathbb{D}^{k,p}$ and the corresponding norms by $\|F\|_{k,p}$. 

As mentioned above, in this paper we will mainly focus on a more pathwise Malliavin derivative taking values in $\bar{\ch}$. Namely we define  the Malliavin derivative in the Cameron-Martin $\bar{\ch}$ space via the isometry $\mathcal{R} $. Precisely, we define $\bar{D}$ such that $\bar{D}F = \mathcal{R}(DF)$. In other words,  for   $h\in \ch$ and a functional $F  $ of $x$ we have   
\begin{eqnarray*}
\bar{D}_{\mathcal{R}(h)}F:=\langle \bar{D}F, \mathcal{R}(h) \rangle_{\bar{\ch}} = \langle DF, h \rangle_{\ch}=: D_{h}F. 
\end{eqnarray*}
This Malliavin derivative can be expressed easily for cylindrical functionals of $x$. Namely  suppose that $F= f(x_{t_{1}}, \dots, x_{t_{\ell}})$ for $f\in C_{p}^{1}(\RR^{\ell})$. According to the definition of $\bar{D}$, for $h\in\ch$ we have
  \begin{eqnarray}\label{e.df.finite}
\langle \bar{D}F, \mathcal{R}(h) \rangle_{\bar{\ch}} = \langle DF, h \rangle_{\ch}=  \sum_{i=1}^{\ell} \partial_{i}f(x_{t_{1}}, \dots, x_{t_{\ell}}) \langle \mathbf{1}_{[0,t_{i}]}, h \rangle_{\ch}
\nonumber
\\
=  \sum_{i=1}^{\ell} \partial_{i}f(x_{t_{1}}, \dots, x_{t_{\ell}}) 
\mathcal{R}(h)(t_{i}).
\end{eqnarray}
 Notice that the   computation in \eqref{e.df.finite} shows that $\langle \bar{D}F, \mathcal{R}(h) \rangle_{\bar{\ch}} $ can be interpreted as an extension in the Fr\'echet derivative of $F $ of the $\mathcal{R}(h) $ direction. Indeed, for the quantity in the right-hand side of \eqref{e.df.finite} we have 
   \begin{equation}\label{a2}
  \frac{d}{d\ep} f(x_{t_{1}}+\ep \mathcal{R}(h)(t_{1}), \dots, x_{t_{\ell}}+\ep \mathcal{R}(h)(t_{\ell})) \Big|_{\ep=0} =    \sum_{i=1}^{\ell} \partial_{i}f(x_{t_{1}}, \dots, x_{t_{\ell}}) 
\mathcal{R}(h)(t_{i}). 
\end{equation}
This pathwise interpretation of Malliavin derivatives is also the one adopted in~\cite{FV}.

In this paper, we denote by $\bar{D}^{k}F$ the $k$th iteration of the Malliavin derivative $\bar{D}$ applied on $F$.  
Also notice that we are considering a $d$-dimensional fBm $x=(x^{1}, \dots, x^{d})$. Therefore, we shall consider partial Malliavin derivatives with respect to each coordinate $x^{i} $ in the sequel. 
Those partial derivatives will be denoed by $\bar{D}^{(i)}$. Then for $\bar{h} = (\bar{h}^{1}, \dots, \bar{h}^{d}) \in \bar{\ch}^{d}$ we write $\bar{D}_{\bar{h}}F = \sum_{i=1}^{d} \langle \bar{D}^{(i)} F, \bar{h}^{i} \rangle_{\bar{\ch}}$.  
For $L\geq 2$ we   denote by $\bar{D}^{L}_{\bar{h}}$ the iterated versions of $\bar{D}_{\bar{h}}$. Namely we set 
\begin{eqnarray}\label{eqn.dlf}
\bar{D}^{L}_{\bar{h}}F = \bar{D}_{\bar{h}}\circ\cdots \circ \bar{D}_{\bar{h}} F. 
\end{eqnarray}
The Sobolev spaces related to the Malliavin derivatives in the Cameron Martin space  are denoted by $\bar{\DD}^{k,p}$ and the corresponding norms are written $\|\cdot\|_{\bar{\DD}^{k,p}}$. 

Let us now review some results on  the Malliavin differentiability of equation \eqref{e1.1}. In the following we assume that the vector fields $V_0,\ldots,V_d$ are at least in $C_b^{3}(\R^m)$ (bounded together with   their derivatives up to order $3$), although later on we'll have to introduce further smoothness conditions in order to estimate higher order Malliavin derivatives.
We shall express the first order Malliavin derivative of $y_{t}$ in terms of the Jacobian $\Phi$ of the equation, which is defined by the relation $\Phi_{t}^{ij}=\partial_{a_j}y_t^{(i)}$. Setting $\partial V_{j}$ for the Jacobian of $V_{j}$ seen as a function from $\R^{n}$ to $\R^{n}$, let us recall that $\Phi$ is the unique solution to the linear equation
\begin{equation}\label{eq:jacobian}
\Phi_{t} = \id_{n} + \int_0^t \partial V_{0} (y_s) \, \Phi_{s} \, ds+
\sum_{j=1}^d \int_0^t \partial  V_j (y_s) \, \Phi_{s} \, dx^j_s,
\end{equation}
The following results hold true:
\begin{proposition}\label{prop:deriv-sde}
Let $y$ be the solution to equation (\ref{e1.1}). Then
for every $i=1,\ldots,m$, $t>0$, and $a \in \mathbb{R}^m$, we have $y_t^{(i)} \in
\mathbb{D} (\ch)$ and
\begin{equation*}
 {D}_s y_t= \lp \Phi_{s,t} V_j (y_s) ,   j=1,\ldots,d\rp , \quad 
0\leq s \leq t,
\end{equation*}
where 
 $\Phi_{t}=\partial_{a} y_t$ solves equation \eqref{eq:jacobian} and $\Phi_{s,t}=\Phi_{t}\Phi_{s}^{-1}$. 
\end{proposition}

Let us now quote the result \cite{CLL}, which gives a useful estimate for moments of the Jacobian of rough differential equations driven by Gaussian processes. Note that this result is expressed in terms of $p$-variations, for which we refer to \cite{FV}.

\begin{proposition}\label{prop:moments-jacobian}
Consider a  fractional Brownian motion $x$ with Hurst parameter $H\in (1/4,1/2]$ and $p>1/H$. Then for any $\eta\ge 1$, there exists a finite constant $c_\eta$ such that the Jacobian $\Phi$ defined by \eqref{eq:jacobian} satisfies:
\begin{equation}\label{eq:moments-J-pvar}
\EE\lc  \Vert \Phi \Vert^{\eta}_{p-{\rm var}; [0,1]} \rc = c_\eta.
\end{equation}
\end{proposition}

 \section{Malliavin derivatives of the Euler scheme}\label{section.der}


The   estimates for  the derivatives of the Euler scheme approximation $y^{n}$  require  a substantial amount of algebraic and analytic efforts. In this  section we focus on the algebraic aspect of the problem. Precisely, we   apply a tree argument to derive a representation for Malliavin derivatives of $y^{n}$. 
This will be useful for our  main bound of  the derivatives of $y^{n}$   in the next section (see Theorem \ref{thm.bd}).

 \subsection{A directed rooted tree}\label{section.tree}
 
 The higher order Malliavin derivatives of the Euler scheme $y^{n}$ are better understood thanks to a tree type encoding. We introduce the necessary notation in this section. Let us start with the definition of rooted trees which will be used in the sequel. 
 
\begin{Def}\label{def.ca}
In the remainder of the paper we consider rooted trees 
$\ca_{N}$ of height $N $ defined recursively as follows: 

\begin{enumerate}[wide, labelwidth=!, labelindent=0pt, label=\emph{(\roman*)}]
\setlength\itemsep{.1in}
\item 
$\ca_{1}$ contains   one branch with length 1 and with root labeled $1$. Namely,  $\ca_{1} = \{(1)\}$.

\item
For each $N\in \NN$ we define $\ca_{N+1}$ such that its first $N$ generations coincides with  $  \ca_{N} $. Its  $(N+1)$-th generation is defined as follows:  
Take a branch $i$  in $\ca_{N}$. We call $\ell^{i}_{1}$ the number of $1$'s in $i$ and we also set $\al_{i}=\ell^{i}_{1}+1$. Then $\ca_{N+1}$ is constructed by adding the branches $(i, 1)$, \dots, $(i, \al_{i})$ to $\ca_{N}$.   Specifically, one can also define $\ca_{N+1}$ recursively as
\begin{equation*}
\big\{ (i, r): i \in \ca_{N},   r=1,\dots, \al_{i} \big\}   . 
\end{equation*}
\end{enumerate}



\end{Def}
\begin{example}\label{ex:A4}
As an example of what Definition \ref{def.ca} can produce, we draw   the $\ca_{4}$ tree in the figure below: 

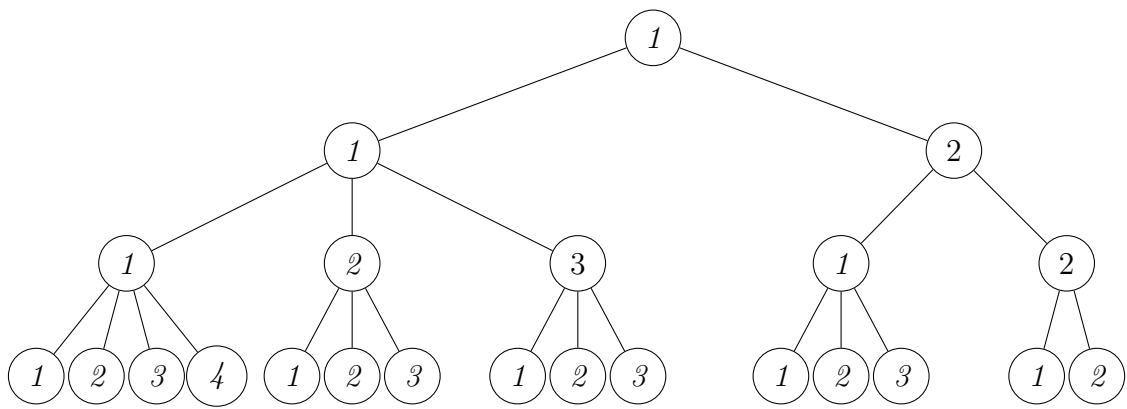
\begin{tikzpicture}[every node/.style={circle,draw},level 1/.style={sibling distance=80mm},level 2/.style={sibling distance=30mm},level 3/.style={sibling distance=8mm}
]
\node {1}
 child 
  {
  node {1} 
     child { node {1} 
           child {node {1}
           }
           child {node {2}}
           child {node {3}}
           child {node {4}} 
     }
     child {
      node {2}
            child {node {1}
                }
            child {node {2}}
            child {node {3}} 
       }
     child { node {$3$} 
            child {node {1}
                 }
            child {node {2}}
            child {node {3}} 
 }
  }
 child 
  { 
 node {$2$} 
     child {node {1}
            child {node {1}
                }
            child {node {2}}
            child {node {3}} 
 } 
     child {node {$2$}
            child {node {1}
                 }
            child {node {2}}
     }
       }
;
\end{tikzpicture}
\end{example}

In the following we introduce some additional  notation about the trees $\ca_{N}$ which will be useful for our future computations. 
 \begin{notation}\label{notation.i}
 With a slight abuse of notation, we will write $\ca_{N}$ for both the tree $\ca_{N}$ and the collection of its branches. 
  For each branch $i $ in  $ \ca_{N}$ we denote   $|i|$ the number of vertices in $i$   and denote $i_{\tau}$   the $\tau$th label in  the branch $i$.   We  set  
 \begin{align*}
 \ell_{1}^{i} = \# \{ i_{\tau}=1: \tau =1,\dots, |i| \}  
 \quad \text{and} \quad
\ell_{r}^{i} = \# \{ i_{\tau}=r: \tau =1,\dots, |i| \} + 1  \quad \text{for } r=2,\dots, |i|.  
\end{align*}
  Also recall that we denote $\al_{i} =  \ell_{1}^{i}+1 $. 
    \end{notation}
    \begin{remark}\label{remark.suml}
    In the sequel we shall use the relation  
    \begin{eqnarray}\label{eqn.suml}
\sum_{r=2}^{\al_{i}} \ell^{i}_{r} =N ,  
\end{eqnarray}
valid for every tree $\ca_{N}$ from Definition \ref{def.ca}. Let us give a brief proof of this fact. According to the definition of $\ca_{N}$, for each branch $i$ the vertices  of $i$ are labeled by the numbers $1,2,\ldots,\al_{i}$. No vertex of $i$ is labeled a number $\al_{i}+1$ or larger. Therefore, for all $r>\al_{i}$ we have
 \begin{eqnarray*}
 \# \{ i_{\tau}=r: \tau =1,\dots, |i| \}=0, 
 \qquad\text{and} \qquad \ell_{r}^{i} =1 .  
\end{eqnarray*}
Moreover, by construction every branch $i$ in $\ca_{N}$ has length $|i|=N$. Thus 
 \begin{eqnarray*}
\sum_{r=1}^{\al_{i}} \# \{ i_{\tau}=r: \tau =1,\dots, |i| \}=|i|=N
 . 
\end{eqnarray*}
Because  $\al_{i}=\ell^{i}_{1}+1= \# \{ i_{\tau}=1: \tau =1,\dots, |i| \}+1$, the above becomes 
\begin{eqnarray*}
\sum_{r=2}^{\al_{i}} \left(\# \{ i_{\tau}=r: \tau =1,\dots, |i| \}\right)+\lp \al_{i}-1\rp =
\sum_{r=2}^{\al_{i}} \left(\# \{ i_{\tau}=r: \tau =1,\dots, |i| \}+1\right) =
N
 . 
\end{eqnarray*}
Otherwise stated, according to Notation \ref{notation.i} we obtain relation \eqref{eqn.suml}.

\end{remark}

 \begin{example}
 Let us follow up on Example \ref{ex:A4}, and see how Notation~\ref{notation.i} works on $\ca_{4}$. Namely $ (1,2,1,3)  $ and $ (1,2,1,1) $ are both branches in $\ca_{4}$. For those branches, the reader can easily check that we have
 \begin{align*}
&
\ell^{(1,2,1,3)}_{1} = 2, \quad 
\ell^{(1,2,1,3)}_{2} = 2 , \quad
\ell^{(1,2,1,3)}_{3} = 2 , \quad
\al_{(1,2,1,3)} = 3 
\\
&
\ell^{(1,2,1,1 )}_{1}=3, \quad
\ell^{(1,2,1,1 )}_{2}=2, \quad
\ell^{(1,2,1,1 )}_{3}=1, \quad
\ell^{(1,2,1,1 )}_{4}=1, \quad
\al_{(1,2,1,1 )}=4 . 
\end{align*}
It is easily checked   that the identity \eqref{eqn.suml} holds for these two branches. Namely  we have          $\ell^{(1,2,1,3)}_{2}+\ell^{(1,2,1,3)}_{3}=4$ and $\ell^{(1,2,1,1)}_{2}+\ell^{(1,2,1,1)}_{3}+\ell^{(1,2,1,1)}_{4}=4$. 
 \end{example}

 In order to state our differentiation rule for Malliavin derivatives, let us also label some notation about partial differentiation in $\RR^{m}$. 
 
 \begin{notation}\label{notationA}
 For a constant  $$A = (A^{p_{1}, 
 \dots,p_{k}};\,\,
  p_{1},\dots ,p_{k} =1\dots,m )\in (\RR^{m})^{\otimes k} = \RR^{m^{k}},$$ and   for $a^{(1)}$, \dots, $a^{(k)} \in \RR^{m}$
  we set 
\begin{eqnarray*}
   \langle A, a^{(1)} \otimes \dots \otimes a^{(k)} \rangle 
&:=
&    \sum_{p_{1}, \dots, p_{k}=1}^{m}  a^{(k)}_{p_{k}} \cdots a^{(1)}_{p_{1}} A^{p_{1},\dots,p_{k}} . 
\end{eqnarray*}

 \end{notation}
 
 \begin{notation}\label{notation.df}
 Let $f: y\to f(y)$ be a  continuous function from $\RR^{m } $ to $ \RR$. We denote by $\partial$ the differential operator from  $   C^{1}(\RR^{m})  $ to $\cl (\RR^{m}, \RR)= \cl (\RR^{m})$. That is 
for $a= (a_{1} , \dots, a_{m} )\in \RR^{m}$, we define $\langle\partial f,  a\rangle = \sum_{i=1}^{m} a_{i} \frac{\partial f}{\partial y^{(i)}} $.  
Note that  the space  $ \cl (\RR^{m})$ can also be identified with   $\RR^{m}$. Namely,   we can  write 
 \begin{eqnarray*}
\partial f := (\partial_{1} f, \dots, \partial_{m} f) := \lp \frac{\partial f}{\partial y^{(1)}}, \dots, \frac{\partial f}{\partial y^{(m)}}  \rp. 
\end{eqnarray*}
One can generalize this notation to higher order derivatives.  Specifically, 
 we denote by $\partial^{k}$ the   differential operator from  $   C^{k}(\RR^{m})  $ to $\cl ((\RR^{m})^{\otimes k}, \RR)$. Otherwise stated for $a^{(1)}$, \dots, $a^{(k)} \in \RR^{m}$ and $y\in \RR^{m}$ we define a vector 
 \begin{equation*}
\partial^{k} f (y) =  \lcl \frac{\partial^{k}f (y)}{ \partial y^{(p_{k})} \cdots\partial y^{(p_{1})}}; \, p_{1},\dots ,p_{k} =1\dots,m \rcl \, ,
\end{equation*}
so that one can write 
\begin{eqnarray}\label{eqn.partial}
 \langle
\partial^{k} f (y) , a^{(1)} \otimes \dots \otimes a^{(k)}\rangle  = \sum_{p_{1}, \dots, p_{k}=1}^{m}  a^{(k)}_{p_{k}} \cdots a^{(1)}_{p_{1}} \frac{\partial^{k}f (y)}{ \partial y^{(p_{k})} \cdots\partial y^{(p_{1})}}. 
\end{eqnarray}
Notice that  for $k=2$, $\partial^{2}f$ can also be identified with the matrix $ ( \frac{\partial^{2}f}{\partial y^{(i)} \partial y^{(j)} } )_{i,j=1,\dots, m} $.

 \end{notation}
 
 \subsection{A differentiation rule}\label{section.drule}
 With the preparation in the previous subsection, 
 we are now ready to state a lemma allowing to compute iterated Malliavin derivatives for a functional of the form $f(F)$, with a smooth enough function $f$ and random variable $F$.  
 Recall that  the Malliavin derivative operator $\bar{D}$ in the Cameron-Martin space $\bar{\ch}$ is introduced  in Section~\ref{subsection.d}. 
 
 \begin{lemma}\label{lem.dfg}
Let $f, g$ be continuous functions in $  C^{N}(\RR^{m})$ and let $F \in \bar{\DD}^{N,2} (\RR^{m})$, where the space $\bar{\DD}^{N,2} $ is introduced in Section \ref{subsection.d}. Then we have the following   identity:
 \begin{align}
\bar{D}_{\bar{h}}^{N}   f(F)  =
&
\sum_{ i \in  \ca_{N}}   
\langle
\partial^{\ell_{1}^{i}} f(F) ,  \bar{D}_{\bar{h}}^{\ell_{2}^{i}} F \otimes\cdots \otimes \bar{D}_{\bar{h}}^{\ell_{\al_{i}}^{i}} F 
\rangle 
\label{eqn.df}
\end{align}
for   $\bar{h}\in \bar{\ch}^{m}$, 
where the sum in the right side of \eqref{eqn.df} runs over the branches of $\ca_{N}$ as specified in Notation~\ref{notation.i}-\ref{notation.df}. As far as the product $f(F) g(F)$ is concerned, we get the following differentiation rule: 
\begin{align}
\bar{D}^{N}_{\bar{h}} (f\cdot g) (F) =& M_{1}(N)
  + M_{2}(N)
, 
 \label{eqn.dfg}
\end{align}
where
\begin{align*}
M_{1}(N) 
&= \sum_{ i \in  \ca_{N}}  
\left\langle 
  (\partial^{\ell_{1}^{i} } f  \cdot g  )(F)    ,  
 \bar{D}_{\bar{h}}^{\ell_{2}^{i}} F \otimes \cdots \otimes \bar{D}_{\bar{h}}^{\ell_{\al_{i}}^{i}} F  
\right\rangle
\\
M_{2}(N) 
&=  \sum_{ i \in  \ca_{N}} 
 \sum_{r=2}^{\al_{i}}    
 \left\langle
\partial^{\ell_{1}^{i}-1 } f(F)     ,  
 \bar{D}_{\bar{h}}^{\ell_{2}^{i}} F \otimes \cdots  \otimes 
 \bar{D}_{\bar{h}}^{\ell_{r}^{i}} g(F) \otimes  \cdots\otimes \bar{D}_{\bar{h}}^{\ell_{\al_{i}}^{i}} F 
 \right\rangle. 
\end{align*}

\end{lemma}
\begin{remark}
Let us check the dimension compatibilities in the right side of \eqref{eqn.df}.  Since $F$ is $\RR^{m}$-valued, according to Notation \ref{notation.df} we have $\partial^{\ell^{i}_{1}} f(F) \in \cl ( (\RR^{m})^{\otimes \ell^{i}_{1}}; \RR ) $. Next each term $\bar{D}^{\ell^{i}_{r}}_{\bar{h}}F$ sits in $\RR^{m}$. Therefore,  $ \bar{D}^{\ell^{i}_{2}}_{\bar{h}}F \otimes \cdots \otimes \bar{D}^{\ell^{i}_{\al_{i}}}_{\bar{h}}F \in     \cl ( (\RR^{m})^{\otimes (\al_{i}-1)} )$. The compatibility of dimensions in \eqref{eqn.df} thus stems from the relation $\al_{i} = \ell^{i}_{1}+1$ (see Notation \ref{notation.i}). Similar considerations are also valid for equation \eqref{eqn.dfg}, taking into account the fact that $\bar{D}^{\ell^{i}_{r}}_{\bar{h}}F\in \RR^{m}$ and $\bar{D}^{\ell^{i}_{r}}_{\bar{h}}g(F)\in  \RR $. 
\end{remark}

\begin{remark}
In order to get a formula for $\bar{D}^{N}_{\bar{h}} f(F) $, we could have invoked some multivariate elaboration of Fa\`a Di Bruno's formula; see   \cite{CS, LP}. However, our tree type formulation is required in order to handle the computations in  Lemma \ref{lem.dhd} below. 
\end{remark}

\begin{remark}
Lemma \ref{lem.dfg} can easily be generalized in three directions:

\noindent\emph{(i)} We have stated \eqref{eqn.df} using the directional derivative $\bar{D}_{\bar{h}}$. The same formula holds true for the function-valued derivative $D_{r} f(F)$. 

 \noindent\emph{(ii)} Instead of the Malliavin derivative, we could have obtained \eqref{eqn.df} as a chain rule for any operator $A$ satisfying a Leibniz type rule of the form $A(f(F)) = f'(F)AF$. 
 
 \noindent\emph{(iii)} Instead of considering iterations of the same operation $A$, that is a formula for $A^{N}(f(F))$, one can obtain a formula like \eqref{eqn.df} for quantities of the form $A_{1}\circ \cdots \circ A_{N} (f(F))$, where each of the operators $A_{j}$ satisfies  $A_{j}(f(F))=f'(F) A_{j}F$.   
\end{remark}

\begin{proof}[Proof of Lemma \ref{lem.dfg}]
 We first show by induction that   \eref{eqn.df} is true.  To this aim, note that $\bar{D}_{\bar{h}}f(F)=\langle \partial f(F),  \bar{D}_{\bar{h}}F\rangle $. On the other hand, by Definition \ref{def.ca} we have $\ca_{1}=\{ (1) \}$ and according to Notation \ref{notation.i} we have  $\ell_{1}=\ell_{2}=1$ and $\al_{(1)}=2$. This concludes \eref{eqn.df} for $N=1$. 
 
 Now suppose that \eref{eqn.df} is true for  $N= L$. This means that $\bar{D}^{L}_{\bar{h}}f(F)$ is equal to the summation of quantities of the form $\langle
 \partial^{\ell_{1}} f(F) ,  \bar{D}_{\bar{h}}^{\ell_{2}} F \otimes \cdots \otimes \bar{D}_{\bar{h}}^{\ell_{\al_{i}}} F \rangle$    which are one-to-one corresponding to   the branches in $\ca_{L}$.  
 Consider   a generic term in this summation and differentiate it in a direction $h\in \bar{\ch}$. Dropping the superscript $i$ in $\ell^{i}_{r}$ for notational sake, we get  
 \begin{align}\label{eqn.di}
&\bar{D}_{\bar{h}} \lp
\langle
\partial^{\ell_{1}} f(F) ,  \bar{D}_{\bar{h}}^{\ell_{2}} F\otimes \cdots\otimes \bar{D}_{\bar{h}}^{\ell_{\al_{i}}} F
\rangle
\rp 
\nonumber
 \\
 &=
 \langle
  \partial^{\ell_{1}+1} f(F) 
 ,  \bar{D}_{\bar{h}}^{\ell_{2}} F \otimes\cdots \otimes  \bar{D}_{\bar{h}}^{\ell_{\al_{i}}} F\otimes \bar{D}_{\bar{h}}  F
 \rangle
 \nonumber
 \\
 &\qquad
 + \langle
 \partial^{\ell_{1}} f(F) ,  \bar{D}_{\bar{h}}^{\ell_{2}+1} F \otimes\cdots\otimes \bar{D}_{\bar{h}}^{\ell_{\al_{i}}} F\rangle+\cdots +
  \langle
 \partial^{\ell_{1}} f(F) ,  \bar{D}_{\bar{h}}^{\ell_{2}} F\otimes \cdots\otimes \bar{D}_{\bar{h}}^{\ell_{\al_{i}}+1} F
 \rangle .  
\end{align}
One can relate relation \eqref{eqn.di} to our tree  Definition \ref{def.ca} in the following way. 
Namely in the recursive step (ii) in Definition \ref{def.ca}, 
from $\ca_{L}$
we have 
 created a new tree by adding labeled offsprings to the branch $i$. Specifically we add the branches $(i,1)$, $(i,2)$, \dots, $(i,\al_{i})$, and we set $\al_{(i, 1)} = \al_{i}+1$ and $\al_{(i, r)} = \al_{i}$ for $r =2, \dots, \al_{i}$. 
This shows that after differentiation,  the term corresponding to the branch $i$ is replaced by $\al_{i}$ terms corresponding to  the branches: $(i,1 )$,  $(i,2)$,\dots,  $(i,\al_{i})$. These are exactly the branches in $\ca_{L+1}$ which overlap with $i$ in the first $|i|$ vertices. 
Here the sum \eqref{eqn.di} can also be written as 
\begin{eqnarray*}
\sum_{j\in K_{i}} \langle \partial^{\ell_{1}^{j}} f (F), \bar{D}^{\ell_{2}^{j}}_{\bar{h}}F \otimes \cdots \otimes  \bar{D}^{\ell_{\al_{j}}^{j}}_{\bar{h}}F  \rangle , 
\end{eqnarray*}
where the branches $j$ sit in a set $K_{i}$ defined by 
\begin{eqnarray*}
K_{i} = \{ (i, r): 1\leq r \leq \al_{i} \}. 
\end{eqnarray*}
Summing all those contributions we get
\begin{eqnarray}\label{eqn.dl1f}
\bar{D}^{L+1}_{\bar{h}}f(F) = 
\sum_{i\in \ca_{L}} \sum_{j\in K_{i}} \langle \partial^{\ell_{1}^{j}} f (F), \bar{D}^{\ell_{2}^{j}}_{\bar{h}}F \otimes \cdots \otimes  \bar{D}^{\ell_{\al_{j}}^{j}}_{\bar{h}}F  \rangle , 
\end{eqnarray}
from which it is easily seen that \eqref{eqn.df} holds up to order $L+1$. This finishes our induction procedure for \eqref{eqn.df}.

   Let us now consider   \eref{eqn.dfg}.  
   First, it is easy to verify that \eref{eqn.dfg} holds for $N=1$. Next suppose that \eref{eqn.dfg} is true for $N=L$.  Let us now differentiate the terms  in $M_{1}(L)$   on the right-hand side of \eref{eqn.dfg}. Still writing $\ell_{r}$ instead of $\ell_{r}^{i}$ for notational sake, we get:
   \begin{align}\label{eqn.h1h2}
&
\bar{D}_{\bar{h}} 
\lp
\langle
(\partial^{\ell_{1}} f\cdot g) (F) ,  \bar{D}_{\bar{h}}^{\ell_{2}} F \otimes\cdots\otimes \bar{D}_{\bar{h}}^{\ell_{\al_{i}}} F
\rangle
\rp
= 
H_{1}^{i}+H_{2}^{i}
,
\end{align}
where 
the term $H^{i}_{1}$ takes care of the differentiation of $g(F)$ in the left side of \eqref{eqn.h1h2}, while $H_{2}^{i}$ corresponds to the differentiation of $\partial^{\ell_{1}}f(F)$ and $\bar{D}^{\ell_{r}}_{\bar{h}} F$. Specifically, we get 
\begin{align*}
H_{1}^{i}
&=
 \langle
   \partial^{\ell_{1}} f(F) 
,  \bar{D}_{\bar{h}}^{\ell_{2}} F \otimes\cdots\otimes \bar{D}_{\bar{h}}^{\ell_{\al_{i}}} F 
\rangle
\cdot \bar{D}_{\bar{h}}  g(F), 
\end{align*}
while $H^{i}_{2}$ is obtained similarly to \eqref{eqn.di} as 
\begin{align*}
H^{i}_{2}&=  
  \langle
   (\partial^{\ell_{1}+1} f\cdot g) (F) 
,  \bar{D}_{\bar{h}}^{\ell_{2}} F \otimes\cdots\otimes \bar{D}_{\bar{h}}^{\ell_{\al_{i}}} F\otimes \cdots\otimes \bar{D}_{\bar{h}}  F 
\rangle
\\
&\quad +
\langle
 (\partial^{\ell_{1}} f\cdot g) (F) ,  \bar{D}_{\bar{h}}^{\ell_{2}+1} F\otimes \cdots\otimes \bar{D}_{\bar{h}}^{\ell_{\al_{i}}} F
 \rangle
+\cdots + 
\langle
(\partial^{\ell_{1}} f\cdot g) (F) ,  \bar{D}_{\bar{h}}^{\ell_{2}} F \otimes\cdots\otimes \bar{D}_{\bar{h}}^{\ell_{\al_{i}}+1} F
\rangle.   
\end{align*}
  Now we follow the same argument as  the one leading to \eqref{eqn.dl1f} to get
  \begin{align}\label{eqn.h2s}
\sum_{i\in \ca_{L}} H^{i}_{2} =  \sum_{i\in \ca_{L+1}}
\langle
(\partial^{\ell_{1}} f\cdot g) (F) ,  \bar{D}_{\bar{h}}^{\ell_{2}} F \otimes\cdots\otimes \bar{D}_{\bar{h}}^{\ell_{\al_{i}}} F
\rangle = M_{1}(L+1) . 
\end{align}
In order to complete the induction proof it remains to show that 
\begin{eqnarray}\label{eqn.dm}
 \sum_{i\in \ca_{L}} H^{i}_{1} + \bar{D}_{\bar{h}} M_{2}(L) = M_{2}(L+1). 
\end{eqnarray}
For this purpose  we consider the  map from $i\in \ca_{L}$ to $(i,1 )$. It is clear that this is a one-to-one mapping.  We conclude from this one-to-one correspondence   that 
  \begin{align}\label{eqn.h1s}
  \sum_{i\in \ca_{L}} H^{i}_{1}
= 
\sum_{j=(i, 1 ): \, i\in \ca_{L}} 
 \langle
   \partial^{\ell_{1}^{j}-1} f  (F) ,  \bar{D}_{\bar{h}}^{\ell_{2}^{j}} F \otimes\cdots\otimes \bar{D}_{\bar{h}}^{\ell_{\al_{j}-1}^{j}} F \otimes \bar{D}_{\bar{h}}^{\ell_{\al_{j}}^{j}} g(F) 
   \rangle,
\end{align} 
 where we recall that according to Notation \ref{notation.i} we have $\ell_{\al_{j}}^{j}=1$ for $j = (i, 1)$.

We turn to  the second summation $M_{2}(L)$ in \eref{eqn.dfg}. 
Note that each term in \eqref{eqn.dfg} corresponds to a couple $(i; r) $ where $i$ is a branch in $\ca_{L}$ and $r \in \{ 2, \dots, \al_{i} \}$ denotes the position for which a term of the form $\bar{D}_{\bar{h}}^{\ell^{i}_{r}}g(F)$ shows up. 
 By differentiating    $M_{2}(L)$ we see that the term corresponding to the couple $(i; r)$ is replaced by the terms corresponding to  $((i, 1 );  r)$, $((i, 2);  r)$, \dots, $((i, \al_{i});  r)$.
 Precisely, we have
 \begin{eqnarray}\label{eqn.b2}
\bar{D}_{\bar{h}} M_{2}(L) = \sum_{(j; r)\in \cb_{2,L}} 
\langle
  \partial^{\ell_{1}-1 } f(F)     ,  
 \bar{D}_{\bar{h}}^{\ell_{2}} F  \otimes\cdots\otimes 
 \bar{D}_{\bar{h}}^{\ell_{r}} g(F)  \otimes \cdots\otimes \bar{D}_{\bar{h}}^{\ell_{\al_{j}}} F
 \rangle ,
\end{eqnarray}
where
\begin{eqnarray*}
\cb_{2,L} = \bigcup_{i\in \ca_{L}}\bigcup_{r=2}^{\al_{i}} \{  ((i, 1 );  r) ,  ((i, 2);  r) , \dots,  ((i, \al_{i});  r)     \} . 
\end{eqnarray*}
Observe that  
in a similar way we can also   write \eqref{eqn.h1s} as   
\begin{eqnarray}\label{eqn.b1}
 \sum_{i\in \ca_{L}} H^{i}_{1}
= 
\sum_{(j;r)\in \cb_{1,L}} 
 \langle
   \partial^{\ell_{1}^{j}-1} f  (F) ,  \bar{D}_{\bar{h}}^{\ell_{2}^{j}} F \otimes\cdots\otimes \bar{D}_{\bar{h}}^{\ell_{\al_{j}-1}^{j}} F \otimes \bar{D}_{\bar{h}}^{\ell_{\al_{j}}^{j}} g(F) 
   \rangle, 
\end{eqnarray}
 where
 \begin{eqnarray*}
\cb_{1,L} = \bigcup_{i\in \ca_{L}}\{((i, 1 );  \al_{i}+1)  \} = \bigcup_{i\in \ca_{L}}\{((i, 1 );  \al_{(i,1)})  \}. 
\end{eqnarray*}

  On the other hand, note  that $\al_{(i, 2)} = \cdots = \al_{(i, \al_{i})} =\al_{i}  $, and  $\al_{(i, 1 )} = \al_{i}+1$. Therefore, we can express the tree $\ca_{L+1}$ as follows: 
\begin{equation}\label{eqn.set}
\bigcup_{i\in \ca_{L+1}}\bigcup_{r=2   }^{\al_{i}}\{ (i;  r)   \} =
\cb_{2,L}
 \cup
\cb_{1,L}.
\end{equation}

Observe that $\cb_{1, L}$ and $\cb_{2, L}$ corresponds to the terms in $H_{2}^{i}$ and $\bar{D}_{\bar{h}}M_{2}(L) $ thanks to \eqref{eqn.b1} and \eqref{eqn.b2}, while   the set $\bigcup_{i\in \ca_{L+1}}\bigcup_{r=2   }^{\al_{i}}\{ (i;  r)   \} $  corresponds to the terms in $M_{2}(L+1)$. We conclude from  identity \eqref{eqn.set} that relation \eqref{eqn.dm} holds.  
 The proof is now complete. 
\end{proof}

\subsection{An expression for the Malliavin derivatives of the Euler scheme}\label{section.xi}

In this subsection we come back to the solution $y$ of our rough differential equation \eqref{e1.1}. However, for notational sake, we shall omit from now the drift term $V_{0} $ in \eqref{e1.1}. Therefore  we are reduced to an equation of the form 
\begin{eqnarray}\label{e.sde}
y_{t} = a + \sum_{i=1}^{d} \int_{0}^{t} V_{i} (y_{s}) dx^{i}_{s}  \, .  
\end{eqnarray}
The corresponding Euler scheme 
  $y^{n}$ (given by \eqref{e4}) can now be expressed as:
\begin{align}\label{eqn.euler}
\delta y^{n}_{t_{k}t_{k+1}} 
=& V(y^{n}_{t_{k}})  x^{1}_{t_{k}t_{k+1}} + \frac12 \sum_{j=1}^{d} \partial V_{j} V_{j} (y^{n}_{t_{k}})  \Delta^{2H}
 ,
\end{align}
 where we recall that  $\Delta = T/n$ and $t_{k} = k\Delta$, where the notation $x^{1}=\delta x$ is introduced in Definition~\ref{def:rough-path}, and where we have used the notation in~\eqref{not:iterated-vector-field} for the quantities $\partial  V_{j}V_{j} (y)$. Notice that for $t\in \ll0, T\rr$, the approximation $y^{n}_{t}$ can also be written as  
 \begin{align}\label{eqn.euler2}
  y^{n}_{ t } 
=& y_{0}+ \sum_{0\leq t_{k}<t}V(y^{n}_{t_{k}})  x^{1}_{t_{k}t_{k+1}} + \frac12 \sum_{0\leq t_{k}<t}\sum_{j=1}^{d} \partial V_{j} V_{j} (y^{n}_{t_{k}})  \Delta^{2H}
 . 
\end{align}
The aim of this section is to find a proper expression for the Malliavin derivatives of $y^{n}$.

\begin{remark}\label{remark.inamaha}
In order to show upper bounds  on  Malliavin derivatives $\|\bar{D}^{L} y^{n}_{t}\|_{\bar{\ch}^{\otimes L}}$ of the Euler scheme  we borrow Inamaha's approach in \cite{Inahama}. 
However, a very special attention to combinatoric issues will have to be paid, due to the fact that we are considering a discrete equation. We have prepared the ground for this in                           Section \ref{section.tree} and  Section     \ref{section.drule}. Furthermore, note that  the uniform continuity in $n$ of Lyons-It\^o's map fails  in our discrete context. Hence    the upper-bound estimates for equation \eqref{eqn.euler}  have to be treated differently for   small and large       
step sizes of the Euler scheme; see Section \ref{section.bound}. 
  \end{remark} 
  
  One of the basic ideas in \cite{Inahama} is to use an independent copy $b$ of the fBm $x $ in order to obtain norms in the Cameron-Martin space $\bar{\ch}$. With this consideration in mind, we now define a family of processes which will be at the heart of our computations of Malliavin derivatives.   
  We start by introducing a family of operators which will be useful for our future definitions. 
  \begin{notation}\label{notation.l}
Let $y^{n}$ be the numerical scheme defined in \eqref{eqn.euler}. Let $f: \RR^{m}\to \RR$ be   a smooth function. 
For each $l=1,\dots, L$, we let   $\xi^{l}_{t} $, $t\in \ll 0,T \rr$ be a process with   values in $ \RR^{m}$. We denote   the process  
       $\xi_{t} = (\xi^{1}_{t},\dots,\xi^{L}_{t}) \in \RR^{Lm}$, $t\in \ll 0,T \rr$. 

Recall that the trees $\ca_{L}$ are introduced in Definition \ref{def.ca}. 
For each $i\in \ca_{L}$ we denote by $c_{L,i}$ and $\tilde{c}_{L,i}$ some    constant depending on $L$ and $i$.  
Also recall our Notation \ref{notationA}-\ref{notation.df}. 
For $\ell_{1} = \ell_{1}^{i}$ corresponding to a branch $i\in \ca_{L}$, we have to consider $\partial^{\ell_{1}}f (y) $  as an element of $\cl ( (\RR^{m})^{\otimes \ell_{1}} ; \RR )$. In addition, still for a branch $i\in \ca_{L}$, owing to the relation $\ell_{1} = \al_{i}-1$ we have 
 $ \xi^{\ell_{2}}_{s} \otimes \cdots \otimes   \xi^{\ell_{\al_{i}}}_{s} \in (\RR^{m})^{\otimes (\al_{i}-1)} = (\RR^{m})^{\otimes\ell_{1}} $. 
 With these elementary algebra considerations in mind, 
  we define the following notation  
   \begin{eqnarray}
\cl^{L}_{\xi,c} f(y^{n}_{t_{k}}) 
&:=
&
 \sum_{ i \in  \ca_{L}}   
c_{L,i}
\langle
\partial^{\ell_{1}} f(y^{n}_{t_{k}} ) , \xi^{\ell_{2}}_{t_{k}} \otimes\cdots \otimes\xi^{\ell_{\al_{i}}}_{t_{k}} 
\rangle
\nonumber
\\
&=
&  \sum_{ i \in  \ca_{L}}   
c_{L,i}\sum_{p_{1}, \dots, p_{\ell_{1} }=1}^{m}  \xi_{t_{k}}^{\ell_{2},p_{1}} \cdots \xi_{t_{k}}^{\ell_{\al_{i}},p_{\ell_{1}}} \frac{\partial^{\ell_{1} }f (y^{n}_{t_{k}})}{ \partial y^{(p_{\ell_{1} })} \cdots\partial y^{(p_{1})}} 
.
\label{eqn.ln}
\end{eqnarray}

Moreover, for smooth functions $f:\RR^{m}\to \RR$ and $g:\RR^{m}\to \RR^{m}$ we     set 
 \begin{align}
\bar{\cl}_{\xi,c}^{L}(\partial f \cdot g) (y^{n}_{t_{k}} ) := 
&
\sum_{ i \in  \ca_{L}}   
c_{L,i}
\langle
(\partial^{\ell_{1}+1 } f \cdot  g)(y^{n}_{t_{k}}  )    ,  
 \xi_{t_{k}}^{\ell_{2}}   \otimes \cdots  \otimes \xi_{t_{k}}^{\ell_{\al_{i}}}   
 \rangle
\nonumber
 \\
 & 
 +  \sum_{ i \in  \ca_{L}} 
\sum_{r=2}^{\al_{i}}
 c_{L,i}   
 \langle 
\partial^{\ell_{1}  } f(y^{n}_{t_{k}} )     ,  
 \xi_{t_{k}}^{\ell_{2}}    \otimes\cdots \otimes
 \cl^{\ell_{r}}_{\xi,c} g(y^{n}_{t_{k}} )  \otimes \cdots  \otimes\xi_{t_{k}}^{\ell_{\al_{i}}} 
 \rangle  
 , 
 \label{eqn.lbn}
 \\
 \tilde{\cl}^{L}_{\xi,c,\tilde{c}} (\partial f\cdot g) (y^{n}_{t_{k}}  )    
 =
 &
\sum_{i\in \ca_{L}} 
\sum_{r=2}^{\al_{i}} 
c_{L,i} 
\langle
\partial^{\ell_{1} } f(y^{n}_{t_{k}}  )     ,  
 \xi^{\ell_{2}}_{_{t_{k}}} \otimes \cdots \otimes
 \cl^{\ell_{r}-1}_{\xi,\tilde{c}} g(y^{n}_{t_{k}}  ) \otimes  \cdots \otimes\xi^{\ell_{\al_{i}}}_{_{t_{k}}}  
 \rangle
    ,
    \label{eqn.ltn}
\end{align}
where $$ \partial^{\ell_{1}+1} f \cdot g (y)
 = 
 \lp\sum_{p_{\ell_{1}+1}=1}^{m} \frac{\partial^{\ell_{1}+1}f (y)}{ \partial y^{(p_{\ell_{1}+1})} \cdots\partial y^{(p_{1})}} g^{p_{\ell_{1}+1}}(y) , \quad p_{1},\dots ,p_{\ell_{1}} =1,\dots,m \rp
  $$ and  
 \begin{eqnarray*}
 \langle
 \partial^{\ell_{1}+1} f \cdot g(y^{n}_{t_{k}} ) , \xi^{\ell_{2}}_{t_{k}} \otimes \cdots \otimes \xi^{\ell_{\al_{i}}}_{t_{k}} 
 \rangle   = \sum_{p_{1}, \dots, p_{\ell_{1}+1}=1}^{m}  \xi_{t_{k}}^{\ell_{2},p_{1}} \cdots \xi_{t_{k}}^{\ell_{\al_{i}},p_{\ell_{1}}} \frac{\partial^{\ell_{1}+1}f (y^{n}_{t_{k}} )}{ \partial y^{(p_{\ell_{1}+1})} \cdots\partial y^{(p_{1})}} g^{p_{\ell_{1}+1}}(y^{n}_{t_{k}} ). 
 \end{eqnarray*}
We also  set $\cl_{\xi,c}^{0}=\bar{\cl}_{\xi,c}^{0}=\id$ and $ \tilde{\cl}_{\xi,c,\tilde{c}}^{0}=\tilde{\cl}_{\xi,c,\tilde{c}}^{-1}=\cl_{\xi,c}^{-1}=\bar{\cl}_{\xi,c}^{-1}=0$. 
\end{notation}

 In order to be able to differentiate our processes of interest in the Malliavin calculus sense, let us label the following regularity assumption on the vector fields $V_{i}$. 

\begin{hyp}\label{hyp:regularity-V}
The vector fields $ V_1,\ldots,V_d$ are $C_b^{(L+2)\vee 3}(\R^m)$ (bounded together with all their derivatives up to order $(L+2)\vee 3$) for    $L\geq 0$.
\end{hyp}  

We can now define a family of paths which will encode the expressions for the Malliavin derivatives of the Euler scheme.

\begin{Def}\label{def.xi}
For $n\geq 1$ we consider the Euler scheme $y^{n}$ given by \eqref{eqn.euler}, and recall that $b$ designates a fBm independent of $x$. Let $\cl^{L}$, $\bar{\cl}^{L}$, $\tilde{\cl}^{L}$ be the operators introduced  in Notation \ref{notation.l}. 
Then for $L\geq 0$ we define a discrete process $\Xi^{L}$
defined for $t=t_{k}$ and taking values in $\RR^{m}$, 
 given similarly to \eqref{eqn.euler2} by   the iterative equation 
\begin{multline}\label{eqn.xin}
\delta \Xi^{L}_{ t_{k}t_{k+1} } 
  =     \cl_{\Xi,c}^{L}   V(y^{n}_{t_{k}})    \delta x_{t_{k}t_{k+1}}
 +
    \cl^{L-1}_{\Xi,\tilde{c}} V(y^{n}_{t_{k}})  \delta b_{t_{k}t_{k+1}} 
\\
+ 
\frac12  \sum_{j=1}^{d}  \bar{\cl}_{\Xi,c}^{L}
\lp \partial V_{j}\cdot V_{j}\rp (y^{n}_{t_{k}}) 
  \Delta^{2H}
+ 
\frac 12  \sum_{j=1}^{d}  \tilde{\cl}_{\Xi,c,\tilde{c}}^{L-1}
\lp \partial V_{j}\cdot V_{j}\rp (y^{n}_{t_{k}}) 
  \Delta^{2H}
,
\end{multline}
or in the integral form  
\begin{multline*}
 \Xi^{L}_{ t } 
  =   \Xi^{L}_{ t_{0} } +  \sum_{t_{0}\leq t_{k}<t}\cl_{\Xi,c}^{L}   V(y^{n}_{t_{k}})    \delta x_{t_{k}t_{k+1}}
 +
   \sum_{t_{0}\leq t_{k}<t}  \cl^{L-1}_{\Xi,\tilde{c}} V(y^{n}_{t_{k}})  \delta b_{t_{k}t_{k+1}} 
\\
+ 
\frac12 \sum_{t_{0}\leq t_{k}<t}\sum_{j=1}^{d}  \bar{\cl}_{\Xi,c}^{L}
\lp \partial V_{j}\cdot V_{j}\rp (y^{n}_{t_{k}}) 
  \Delta^{2H}
+ 
\frac 12  \sum_{t_{0}\leq t_{k}<t}\sum_{j=1}^{d}  \tilde{\cl}_{\Xi,c,\tilde{c}}^{L-1}
\lp \partial V_{j}\cdot V_{j}\rp (y^{n}_{t_{k}}) 
  \Delta^{2H}
, 
\end{multline*}
where $t_{0}\in \ll0,T\rr$ is the initial time of the iteration equation  and  $c=(c_{L,i},i\in\ca_{L})$ and $c=(\tilde{c}_{L,i},i\in\ca_{L})$ are some constants.  

 Note that we apply  $\cl^{L}_{\Xi,c}$ to every component of $V$ (i.e. $f=V^{i}_{j}$ for each $i$ and $j$) in order to get a $\RR^{m}$-valued element $\cl_{\Xi,c}^{L} V(y^{n}_{t_{k}})\delta x_{t_{k}t_{k+1}}$ in the right-hand side of \eqref{eqn.xin}.  
Precisely, we have   $\cl^{L}_{\Xi,c}   V(y) = ( \cl^{L}_{\Xi,c}  V^{i}_{j} , i=1\dots, m, j=1,\dots, d)$.  $ \bar{\cl}^{L}_{\Xi,c} (\partial V_{j} \cdot V_{j}) (y ) $ and $ \tilde{\cl}^{L}_{\Xi,c,\tilde{c}} (\partial V_{j}\cdot V_{j}) (y ) $ should be interpreted  in the  same way.

\end{Def}
 
 We now state our general expression for the Malliavin derivatives of the Euler scheme. In the following, for conciseness we   drop  the subscript $(\Xi,c)$ of $\cl^{L}_{\Xi,c}$ and simply write $\cl^{L}$. This simplification is also  applied  to $\bar{\cl}^{L}_{\Xi,c}$  
 and $\tilde{\cl}^{L}_{\Xi,c,\tilde{c}}$.  
 \begin{lemma}\label{lem.dhd}
For $n\geq 1$ let $y^{n}$ be the Euler scheme defined by \eqref{eqn.euler2}. Assume that Hypothesis \ref{hyp:regularity-V} holds for     $L\geq 1$.  Recall that a notation $\bar{D}_{\bar{h}}$ has been introduced in Section \ref{subsection.d} for the Malliavin derivative with respect to the fBm $x$. We also write $\hat{{D}}_{\bar{h}}$ for the directional derivative with respect to the independent fBm $b$. Let $\Xi^{L}$ be the process   introduced    in Definition \ref{def.xi} with $c_{L,i} =\tilde{c}_{L,i} = \frac{\ell_{2}!\cdots \ell_{\al_{i}}!}{L !} $ for all $i\in \ca_{L}$.
 Then 
 for all $t\in \ll0,T\rr$
  the iterated derivative \eqref{eqn.dlf} of $y^{n}_{t} $   can be expressed as 
\begin{align}\label{eqn.dxidy}
   \bar{D}^{L}_{\bar{h}} y^{n}_{t}=\hat{{D}}^{L}_{\bar{h}}  \Xi^{L}_{t} .  
\end{align}
\end{lemma}
\begin{proof}
According to Definition \ref{def.xi} and recalling our convention $\cl^{-1}=\tilde{\cl}^{-1}=0$, it is straightforward to see that $\Xi^{0} = y^{n}$. 

For $L\geq 1$, consider the process $\Xi^{L}$ defined by \eqref{eqn.xin}. Our next endeavor is to find a difference equation satisfied by $\hat{{D}}^{L}\Xi^{L}$. To this aim we differentiate the terms in the right-hand side of \eqref{eqn.xin}. Note   that $y^{n}$ does not depend on $b$ and therefore $\hat{{D}} \lc \partial^{\ell_{1}} V(y^{n}_{t_{k}}) \rc=0$. 
We now prove that $\Xi^{\ell}$ belongs to the $\ell$-th Wiener chaos of $b$, which will be denoted  by $\mathcal{K}^{b}_{\ell}$ (similarly to that of $x$ in  Section \ref{subsection.d}). This can be done recursively on $L$ using relation~\eqref{eqn.xin}. Namely   we assume that $\Xi^{p}_{t_{k}} \in \mathcal{K}^{b}_{p}$ for all $p\leq L-1$.  It can be checked that the terms on the right-hand side of \eqref{eqn.xin} belongs to $\ck^{b}_{L}$. For sake of conciseness, let us focus on  the following  term of \eqref{eqn.xin} (the other terms being left to the patient reader): 
\begin{eqnarray}\label{eqn.lvb}
\cl^{L-1}_{\Xi, \tilde{c}} V(y^{n}_{t_{k}}) \delta b_{t_{k}t_{k+1}}
=
\sum_{i\in \ca_{L-1}} \tilde{c}_{L-1, i} 
\left\langle
\partial^{\ell_{1}}V(y^{n}_{t_{k}}) , \, \Xi^{\ell_{2}}_{t_{k}} \otimes \cdots \otimes  \Xi^{\ell_{\al_{i}}}_{t_{k}}
\right\rangle 
\delta b_{t_{k}t_{k+1}}. 
\end{eqnarray}
  By the induction assumption  the generic term on the right-hand side of \eqref{eqn.lvb} is in the chaos 
\begin{eqnarray}\label{eqn.hbl}
\mathcal{K}^{b}_{\sum_{p=2}^{\al_{i}}\ell_{p}+1} = \mathcal{K}^{b}_{(L-1)+1} = \mathcal{K}^{b}_{L}, 
\end{eqnarray}
where we have invoked Remark \ref{remark.suml} for the second equation. Relation \eqref{eqn.hbl} thus proves that $\Xi^{\ell}_{t_{k}}\in \mathcal{K}^{b}_{\ell}$ by induction. In particular, $\hat{{D}}_{\bar{h}}^{\ell'} \Xi^{\ell}=0$ if $\ell'>\ell$.

In order to differentiate the right-hand side of \eqref{eqn.xin}, we need to differentiate generic terms of the form $\Xi^{\ell_{2}}_{t_{k}}\otimes \cdots \otimes \Xi^{\ell_{\al_{i}}}_{t_{k}}$ for $i \in \ca_{L}$. It is easily seen that 
\begin{eqnarray*}
\hat{{D}}^{L}_{\bar{h}}\lp \Xi^{\ell_{2}}_{t_{k}} 
\otimes \cdots \otimes \Xi^{\ell_{\al_{i}}}_{t_{k}}\rp
=
\sum_{(p_{2},\dots,p_{\al_{i}}):
p_{2}+\dots+p_{\al_{i}}=L } 
\frac{L!}{p_{2}!\cdots p_{\al_{i}}!}
\hat{D}^{p_{2}}_{\bar{h}}\Xi^{\ell_{2}}_{t_{k}} 
\otimes \cdots \otimes 
\hat{D}^{p_{\al_{i}}}_{\bar{h}}\Xi^{\ell_{\al_{i}}}_{t_{k}} . 
\end{eqnarray*}
However, if $(p_{2},\dots,p_{\al_{i}}) \neq (\ell^{i}_{2}, \dots, \ell^{i}_{\al_{i}})$, at least one of the $\ell^{i}_{j}$ will be larger than $p_{j}$, yielding a null contribution. Therefore the only surviving term in the sum above is  
\begin{eqnarray}\label{eqn.dxidxi}
\hat{{D}}_{\bar{h}}^{L} [  \Xi^{\ell_{2}}_{t_{k}} \otimes \cdots \otimes \Xi^{\ell_{\al_{i}}}_{t_{k}} ]  =  
\frac{L !}{\ell_{2}!\cdots \ell_{\al_{i}}!}\cdot
\hat{{D}}_{\bar{h}}^{\ell_{2}}\Xi^{\ell_{2}}_{t_{k}} \otimes \cdots \otimes\hat{{D}}_{\bar{h}}^{\ell_{\al_{i}}}\Xi^{\ell_{\al_{i}}}_{t_{k}}
. 
\end{eqnarray}
The reader is  referred  to \cite{Inahama} for more details about the above  computation. 
 Note that the number of ways to assign the $L$ operators $\hat{{D}}_{\bar{h}}$,\dots, $\hat{{D}}_{\bar{h}}$  into  groups of sizes $\ell_{2}$,\dots, $\ell_{\al_{i}}$ is $\frac{L !}{\ell_{2}!\cdots \ell_{\al_{i}}!}$, which explains the multiplicative constant in the equation. 
 
 With \eqref{eqn.dxidxi} in hand, we are now ready to   differentiate the right-hand side of \eqref{eqn.xin}. For the first term, using definition \eqref{eqn.ln} of  $\cl^{L}$
 we get
 \begin{align*}
\hat{{D}}_{\bar{h}}^{L}  \cl^{L} V(y^{n}_{t_{k}}) = \sum_{ i \in  \ca_{L}}   
\langle
\partial^{\ell_{1}} V(y^{n}_{t_{k}}),  
 \hat{{D}}_{\bar{h}}^{\ell_{2}}\Xi^{\ell_{2}}_{t_{k}} \otimes \cdots \otimes\hat{{D}}_{\bar{h}}^{\ell_{\al_{i}}}\Xi^{\ell_{\al_{i}}}_{t_{k}}
 \rangle.  
\end{align*}
Along the same lines and resorting to \eqref{eqn.lbn} for the definition of $\bar{\cl}^{L}$, it is easily checked that  
\begin{eqnarray*}
\hat{{D}}_{\bar{h}}^{L}\bar{\cl}^{L}(\partial V \cdot V) (y^{n}_{t_{k}}) &= 
&
\sum_{ i \in  \ca_{L}}  
\langle 
(\partial^{\ell_{1}+1 } V \cdot V)(y^{n}_{t_{k}})    ,  
 \hat{{D}}_{\bar{h}}^{\ell_{2}}\Xi^{\ell_{2}}_{t_{k}}  \otimes \cdots \otimes  \hat{{D}}_{\bar{h}}^{\ell_{\al_{i}}}\Xi^{\ell_{\al_{i}}}_{t_{k}}
 \rangle  
\nonumber
 \\
& & +
  \sum_{ i \in  \ca_{L}}     
\sum_{r=2}^{\al_{i}} 
\langle
\partial^{\ell_{1}  } V(y^{n}_{t_{k}})     ,  
 \hat{{D}}_{\bar{h}}^{\ell_{2}}\Xi_{t_{k}}^{\ell_{2}}   \otimes \cdots \otimes
 \hat{{D}}_{\bar{h}}^{\ell_{r}}\cl^{\ell_{r}} V(y^{n}_{t_{k}})  \otimes \cdots\otimes  \hat{{D}}_{\bar{h}}^{\ell_{\al_{i}}}\Xi_{t_{k}}^{\ell_{\al_{i}}}  
 \rangle 
  .
\end{eqnarray*}
Similarly, for the second term on the right-hand side of \eqref{eqn.xin}, we end up with 
\begin{align*}
\hat{{D}}^{L}_{\bar{h}} \cl^{L-1} V(y^{n}_{t_{k}})  \delta b_{t_{k}t_{k+1}} = L  \sum_{i\in \ca_{L-1}} 
\langle
\partial^{\ell_{1}} V(y^{n}_{t_{k}}),  \hat{{D}}^{\ell_{2}}_{\bar{h}}\Xi^{\ell_{2}}_{t_{k}} \otimes\cdots\otimes \hat{{D}}^{\ell_{\al_{i}}}_{\bar{h}} \Xi^{\ell_{\al_{i}}}_{t_{k}}  \otimes \delta \bar{h}_{t_{k}t_{k+1}}
\rangle.    
\end{align*}
Note also that   the fourth term on the right-hand side of \eref{eqn.xin} is in the $(L-2)$th chaos of $b$ and thus has zero   $\hat{{D}}^{L}$ derivative.
 Differentiating both sides of \eqref{eqn.xin} and taking into account the above computations,  we have thus obtained 
 \begin{align}\label{eqn.dxin}
\hat{{D}}^{L}_{\bar{h}} 
\Xi^{L}_{t} =& \sum_{0\leq t_{k}<t}
\sum_{ i \in  \ca_{L}}   
\langle\partial^{\ell_{1}} V(y^{n}_{t_{k}}) , 
 \hat{{D}}_{\bar{h}}^{\ell_{2}}\Xi^{\ell_{2}}_{t_{k}}\otimes \cdots \otimes\hat{{D}}_{\bar{h}}^{\ell_{\al_{i}}}\Xi^{\ell_{\al_{i}}}_{t_{k}}
 \rangle
 \delta x_{t_{k}t_{k+1}} 
 \nonumber
 \\
 &+  L \sum_{0\leq t_{k}<t}
  \sum_{i\in \ca_{L-1}} 
  \langle
  \partial^{\ell_{1}} V(y^{n}_{t_{k}}),  \hat{{D}}^{\ell_{2}}_{\bar{h}}\Xi^{\ell_{2}}_{t_{k}} \otimes\cdots\otimes \hat{{D}}^{\ell_{\al_{i}}}_{\bar{h}} \Xi^{\ell_{\al_{i}}}_{t_{k}} \otimes  \delta \bar{h}_{t_{k}t_{k+1}}\rangle
\\
&+ \frac12 \sum_{0\leq t_{k}<t}\sum_{j=1}^{d} 
\sum_{ i \in  \ca_{L}} 
\Big(   \langle
(\partial^{\ell_{1}+1 } V_{j}  \cdot V_{j})(y^{n}_{t_{k}}) ,      
 \hat{{D}}_{\bar{h}}^{\ell_{2}}\Xi^{\ell_{2}}_{t_{k}} \otimes  \cdots  \otimes \hat{{D}}_{\bar{h}}^{\ell_{\al_{i}}}\Xi^{\ell_{\al_{i}}}_{t_{k}}  \rangle
\nonumber
 \\
 & 
\qquad\qquad
 +\sum_{r=2}^{\al_{i}}   
 \langle
 \partial^{\ell_{1}  } V_{j}(y^{n}_{t_{k}})     ,  
 \hat{{D}}_{\bar{h}}^{\ell_{2}}\Xi_{t_{k}}^{\ell_{2}}   \otimes \cdots \otimes
 \hat{{D}}_{\bar{h}}^{\ell_{r}}\cl^{\ell_{r}} V_{j}(y^{n}_{t_{k}})  \otimes \cdots  \otimes\hat{{D}}_{\bar{h}}^{\ell_{\al_{i}}}\Xi_{t_{k}}^{\ell_{\al_{i}}} 
 \rangle
\Big)
  \Delta^{2H}. \nonumber
\end{align}
Let us now differentiate $y^{n}_{t}$ according to its definition \eqref{eqn.euler2}. To this aim resorting to the fact that $\delta x_{t_{k}t_{k+1}}$ is in the first chaos of $x$,  we get   
  \begin{multline}\label{eqn.dly} 
 \bar{D}_{\bar{h}}^{L} y^{n}_{ t } 
=  \sum_{0\leq t_{k}<t}\bar{D}_{\bar{h}}^{L}   V(y^{n}_{t_{k}})    \delta x_{t_{k}t_{k+1}}
 +
  L\sum_{0\leq t_{k}<t}  \bar{D}^{L-1}_{\bar{h}}V(y^{n}_{t_{k}})  \delta \bar{h}_{t_{k}t_{k+1}} 
\\
+ 
\frac12 \sum_{0\leq t_{k}<t}\sum_{j=1}^{d} \bar{D}^{L}_{\bar{h}}\lc
\lp \partial V_{j}\cdot V_{j}\rp (y^{n}_{t_{k}}) 
\rc
  \Delta^{2H}
 . 
\end{multline}
Next we differentiate the terms $V(y^{n}_{t_{k}})$ and $(\partial V_{j}\cdot V_{j})(y^{n}_{t_{k}})$ thanks to   
Lemma \ref{lem.dfg}. It is readily checked that we get exactly the same expression as \eqref{eqn.dxin}. This shows our claim \eqref{eqn.dxidy} and finishes our proof. 
\end{proof}

\section{Upper-bound estimates of Malliavin derivatives of the Euler scheme}
\label{section.bd}

In Section \ref{subsection.d} we have   recalled some known results  on the Malliavin derivative of the solution $y$ to \eqref{e1.1}. Note that upper bounds for higher order derivatives of $y$ are obtained in \cite{CHLT, Inahama}. With the preparation in Section \ref{section.der} we are ready to extend those estimates to the Euler scheme approximations $y^{n}$.

  \subsection{Some auxiliary results}

 This subsection is dedicated to     some necessary auxiliary      results. 
  Throughout the subsection we fix an integer $N>0$. 
  Recall that    $\Xi^{L}$ is  defined by the iterative equation  in Definition \ref{def.xi}. 
   We start by    introducing a process related to $\Xi^{L}$, $L=0,1,\dots,N$. 
  In the following we  assume  that Hypothesis \ref{hyp:regularity-V}  holds with       $L $ replaced by $N$. Namely, we assume that $V\in C^{(N+2)\vee 3}(\RR^{m})$. 
 
   \begin{Def}\label{def.p}
  For each $s\in [0,T]$    we define  $P^{L}_{s}$ to be    the maximum among the quantities of the form $|\Xi^{i_{1}}_{s}| \times \cdots\times |\Xi^{i_{N_{0}}}_{s}| $ with   $N_{0}>0$ such that $i_{1}, \dots, i_{N_{0}}\in \{1,\dots, L\}$ and  $i_{1}+\cdots +i_{N_{0}}\leq L$. We also  define $P^{0}\equiv 1$ and $P^{-1}\equiv 0$. 
 \end{Def}
 The following result   follows immediately  from the definition of $P^{L}$:
  \begin{lemma}\label{lem.pn2}
 The following  three inequalities hold for   $s\in \ll0,T\rr$:  
 \begin{align*}
P^{L}\geq P^{L'} ,  \quad \quad P^{L}_{s}\times P^{L'}_{s}\leq P^{L+L'}_{s} 
\quad \text{ for } L\geq  L'\geq 0,  
\quad \text{and} \quad
 |\Xi^{L}_{s}|\leq P^{L}_{s}
\quad
\text{ for } L\geq 1. 
\end{align*}
 \end{lemma}

 Let us now fix some    constants which we will make extensive use of: 
 For $f\in C^{N+2}_{b}$, $N\in\NN$ we set
 \begin{align}\label{def.cv}
 C^{0}_{ f} = \sup_{\tau\leq N+2} \|\partial_{\tau} f \|_{\infty}
, 
\end{align}
where $\| \cdot \|_{\infty}$ denotes the sup norm for continuous functions. 
Also recall that for $i\in \ca_{L}$ the constants  $\al_{i}$ and $c_{L,i}$   are  defined respectively in Notation \ref{notation.i} and Lemma  \ref{lem.dhd},   and $C^{0}_{ V}$ is defined in \eref{def.cv}.  
 Then for $L=1,2,\dots, N$ we define
 \begin{align}
   C^{1}_{L,V}   =
 \sum_{i\in \ca_{L}} c_{L, i}C^{0}_{ V}  ,   
\quad
C^{1}_{0,V}=C^{0}_{V}, \quad
C^{1}_{-1, V}=0, 
\label{eqn.c1}
\\
  C^{2}_{L,V}  = \sum_{i\in \ca_{L}} c_{L, i}C^{0}_{V} \lp C^{0}_{V}  + \sum_{r=2}^{\al_{i}} C^{1}_{\ell_{r},V}   \rp  , 
\quad
C^{2}_{0,V}=2(C^{0}_{V} )^{2} , \quad C^{2}_{-1,V}=0, 
\label{eqn.c2}
 \\
   C^{3}_{L,V} = \sum_{i\in \ca_{L}} \sum_{r=2}^{\al_{i}}c_{L, i}C^{0}_{V}C^{1}_{\ell_{r}-1,V},      \quad C^{3}_{0,V}=0,      \quad C^{3}_{-1,V}=0
   .  
   \label{eqn.c3}
\end{align}
We will also resort to the following constants:   
\begin{align}\label{eqn.k12}
K_{1}^{L} =    
 C^{1}_{L,V}+C^{1}_{L-1,V}+1 , 
 \qquad
 K_{2}^{L} = 
    C^{2}_{L,V}+C^{2}_{L-1,V}+ C^{3}_{L,V}+C^{3}_{L-1,V}+1 .
\end{align}

Next we introduce a family of sets in the following way, for $L\geq 1$: 
\begin{align*}
S_{L} = \{(L', \ell_{1}, \dots, \ell_{L'}): L' \in \NN_{+}, \ell_{1}, \dots, \ell_{L'}  \in \NN_{+} , \,  
\ell_{1}+ \dots+ \ell_{L'}=L\}. 
\end{align*}
Related to this definition, we define another family of constants:  
\begin{align}\label{eqn.c4}
C_{L}^{4} =\max_{(L', \ell_{1}, \dots, \ell_{L'})  \in S_{L} }  
   2^{L+1}\times  K_{1}^{\ell_{1}}\times\cdots\times K_{1}^{\ell_{L'}} 
  ,
\end{align}
and 
\begin{align}\label{eqn.c5}
C_{f,L}^{5} = 2 \sum_{i\in \ca_{L}} c_{L,i} C^{0}_{f}  
\lp 
K^{1}_{1} + 
\sum_{r=2}^{\al_{i}} K^{\ell_{r}}_{1}  \rp (1+C_{L}^{4}), 
\qquad
C_{ f,0}^{5} = C^{0}_{f}K^{0}_{1}. 
\end{align}
 
 \begin{remark}
The constants   introduced above will appear  in our proof for   the upper bound  of $\Xi^{L}$. We will see that  because our proof is an induction argument, it is important to keep track of these constants.  
\end{remark}

With this additional notation, 
in the following we derive an upper-bound estimate for the product of $\Xi^{L}$.  
\begin{lemma}\label{lem.dXi}
Let $\omega $ be a control function on $\ll0,T\rr$. 
Recalling our definition \eqref{eq:def-delta} for the operator $\delta$, 
let $(s,u) \in \cs_{2}(\ll 0,T\rr)$ be such that  
\begin{align}\label{eqn.cdxi}
\omega (s,u)^{1/p} \leq 1/2, 
\qquad \text{and} \qquad 
|\delta \Xi^{L}_{su}|\leq K^{L}_{1} P^{L}_{s} \omega (s,u)^{1/p},  \quad L=1,\dots, N.
\end{align}
 Let  $L_{0}\leq N$ and        $(N', \ell_{1}, \dots, \ell_{N'})\in S_{L_{0}}$.  Then  we have
\begin{align}\label{eqn.dxx}
|\delta  \lp
    \Xi^{\ell_{ 1}}
   \otimes \cdots\otimes\Xi^{\ell_{N'}}   
 \rp_{su} | \leq  
 C_{L_{0} }^{4} \cdot P^{L_{0}}_{s} \cdot 
 \omega(s,u)^{1/p} . 
\end{align} 

\end{lemma}
\begin{proof} 
 We first note that a straightforward computation shows that $\delta  \lp
    \Xi^{\ell_{1}}
   \otimes \cdots\otimes\Xi^{\ell_{N'}}   
 \rp_{su}$ is equal to the summation of products of the quantities of the forms $\Xi^{\ell_{r}}_{s}$  and $\delta \Xi^{\ell_{r}}_{su}$ with $r=1,\dots, N'$.   
 Apply  Lemma \ref{lem.pn2} to $\Xi^{\ell_{r}}$ and   condition \eref{eqn.cdxi} to   $|\delta\Xi^{\ell_{r}}_{st}|$. 
 Also take   into account that $\omega (s,u)^{1/p}\leq 1/2$, according to \eqref{eqn.cdxi}.  We obtain that each product in the summation is bounded by 
 \begin{align*}
 \lp
 \max_{(N', \ell_{1}, \dots, \ell_{N'})  \in S_{L_{0}} }  
       K_{1}^{\ell_{1}}\times\cdots\times K_{1}^{\ell_{N'}}
       \rp
        \cdot
   P^{L_{0}}_{s} \cdot 2\omega (s,u)^{1/p}. 
\end{align*}
 Note that  there are at most  $2^{L_{0}}-1$ terms in the summation. Hence owing to   the definition of $C^{4}_{L_{0}}$ in \eref{eqn.c4}   we obtain the desired estimate \eref{eqn.dxx}.  
\end{proof}
Following is   an estimate for the Euler scheme $y^{n}$:   
\begin{lemma}\label{lem.dV}
Let $\omega $ be a control function on $\ll0,T\rr$ and consider the Euler scheme in \eqref{eqn.euler}. 
Let $(s,u) \in \cs_{2}(\ll 0,T\rr)$ be such that  
\begin{align}\label{eqn.cdy}
 |\delta y^{n}_{su}|\leq K^{0}_{1}  \omega (s,u)^{1/p} .
\end{align}
Then   for $0\leq L\leq N$
and recalling that $ C_{f}^{0} $ is defined by \eqref{def.cv},
 we have: 
\begin{align*}
 |
 \delta\lp    \partial^{L} f (y^{n}_{\cdot})
  \rp_{su}|\leq C^{0}_{f} K^{0}_{1} \omega (s,u)^{1/p}.  
\end{align*}
 
\end{lemma}
\begin{proof}
The lemma follows immediately  from the mean value theorem and the condition~\eref{eqn.cdy}.  
\end{proof}
Recall that we have defined the solution to \eqref{e.sde} as the controlled process   in \eqref{eq:dcp-Davie}. The following lemma improves our Lemma \ref{lem.dV} when $y$ is a discrete controlled process. 
\begin{lemma}\label{lem.ddV}
Let $\omega $ be a control function on $\ll0,T\rr$. 
Let $(s,u) \in \cs_{2}(\ll 0,T\rr)$ be such that  
\begin{align}\label{eqn.cddy}
 |\delta y^{n}_{su}|\leq K^{0}_{1}  \omega (s,u)^{1/p}
 \qquad
 \text{and}
 \qquad
  |\delta y^{n}_{su} - V(y^{n}_{s})\delta x_{su} |\leq K^{0}_{2}  \omega (s,u)^{2/p}.
\end{align} 
Then the  following relation  holds true for $0\leq L\leq N$: 
\begin{align*}
 |
 \delta\lp    \partial^{L} V (y^{n}_{\cdot})
  \rp_{su}
  - (\partial^{L+1} VV )(y^{n}_{s}) \delta x_{su}  
  |\leq  C^{0}_{V} 
  \lp K^{0}_{1}K^{0}_{1} +K^{0}_{2} \rp \omega (s,u)^{2/p}   .  
\end{align*}
\end{lemma}
\begin{proof}
The lemma follows from the application of an obvious second order Taylor expansion, as well as the conditions in \eref{eqn.cddy}. Note that here we need the fact that $\|\partial^{N+2}V\|_{\infty}\leq C^{0}_{V}$ when $L=N$.  The details are omitted for sake of conciseness. 
\end{proof}
Recall that $\cl^{L}_{ \Xi, c}$, $\bar{\cl}^{L}_{ \Xi, c}$, $\tilde{\cl}^{L}_{\Xi, c,\tilde{c}}$ are defined in Notation \ref{notation.l}. For the sake of simplicity we will drop the subscript and write $\cl^{L} $, $\bar{\cl}^{L} $, $\tilde{\cl}^{L} $ in the following series of lemmas.  
\begin{lemma}\label{lem.lbd}
Recall that $y^{n}$ is the numerical scheme given by \eqref{eqn.euler}, and that Hypothesis~\ref{hyp:regularity-V} holds true. For $L\geq 0$ let $C^{1}_{L,V}$, $C^{2}_{L,V}$ and $C^{3}_{L,V}$ be the constants defined by \eqref{eqn.c1}-\eqref{eqn.c3} and recall that $P^{L}_{s}$ is introduced in Definition \ref{def.p}. Then the following holds true for all $s\in \ll 0,T\rr$ and $L\geq 0$:
\begin{align}
| \cl^{L}  V(y^{n}_{s}) |
\leq 
  C^{1}_{L,V} P^{L}_{s},   
\label{eqn.lnbd}
\\
|
\bar{\cl}^{L}(\partial V V (y^{n}_{s}))| \leq 
  C^{2}_{L,V} P^{L}_{s}, 
 \label{eqn.lbnbd}
\\
| \tilde{\cl}^{L} (\partial VV) (y^{n}_{s})  |\leq 
  C^{3}_{L,V} P^{L}_{s}
 .  
\label{eqn.ltnbd}
\end{align}
\end{lemma}
\begin{proof}
An application of Lemma \ref{lem.pn2} to \eref{eqn.ln} yields  
\begin{align*}
| \cl^{L} V(y^{n}_{s}) |
\leq 
\sum_{i\in \ca_{L}} c_{L, i}C^{0}_{V} P^{L}_{s}. 
\end{align*}
Relation   \eref{eqn.lnbd} then follows immediately from  the definition of $ C^{1}_{L,V}$ in \eref{eqn.c1}. 
Now   apply Lemma \ref{lem.pn2}  to \eref{eqn.lbn} as before,  and then apply \eref{eqn.lnbd}
in order to handle the terms $\cl^{\ell_{r}}g(y^{n}_{t_{k}})$ in the right-hand side of \eqref{eqn.lbn}.  We obtain:  
\begin{align*}
|
\bar{\cl}^{L}(\partial V V (y^{n}_{s}))| \leq 
\sum_{i\in \ca_{L}} c_{L, i}C^{0}_{V} 
\left(C^{0}_{V}P^{L}_{s} 
+ \sum_{r=2}^{\al_{i}} C^{1}_{\ell_{r},V} P^{L}_{s} \right)  . 
\end{align*}
Hence resorting to the definition of $C^{2}_{L,V}$ in  \eref{eqn.c2} we obtain relation \eref{eqn.lbnbd}. 
The last relation~\eref{eqn.ltnbd} can be shown in a similar way. We have  
\begin{align*}
| \tilde{\cl}^{L} (\partial VV) (y^{n}_{s})  |\leq 
\sum_{i\in \ca_{L}} \sum_{r=2}^{\al_{i}}c_{L, i}C^{0}_{V}C^{1}_{\ell_{r}-1,V} P^{L}_{s} . 
\end{align*}
  Relation  \eref{eqn.ltnbd} then follows from the definition of $C^{3}_{L,V}$ in \eref{eqn.c3}. 
\end{proof}
In the following we consider the    increments of processes in Lemma \ref{lem.lbd}.  
\begin{lemma}\label{eqn.dlpvv}
Let $\omega $ be a control function on $\ll0,T\rr$. 
Recall that we write
$\cl^{L}$, $\bar{\cl}^{L}$, $\tilde{\cl}^{L}$ for 
 $\cl^{L}_{\Xi, c}$, $\bar{\cl}^{L}_{\Xi, c}$, $\tilde{\cl}^{L}_{\Xi, c,\tilde{c}}$. 
Assume that \eref{eqn.cdxi}
and \eqref{eqn.cdy} hold (notice that \eqref{eqn.cdxi} for $L=0$ in fact implies \eqref{eqn.cdy})
  for  some $(s,u)\in \cs_{2} (\ll 0,T \rr) $. 
Then we have the following relations $L\geq 0$: 
\begin{align}
|\delta\lp
     \cl^{L} f (y^{n}_{\cdot})
    \rp_{su}|
    \leq C^{5}_{f,L} P^{L}_{s}    
 \omega(s,u)^{1/p}   ,
 \label{eqn.dlnbd}
 \\
 \Big|\delta\lp
      \bar{\cl}^{L} ( \partial VV ) (y^{n}_{\cdot})
    \rp_{su}\Big|
    \leq  C^{6}_{ V,L} P^{L}_{s}    
 \omega(s,u)^{1/p}  ,
 \label{eqn.dlbnbd}
   \\
 \Big|\delta\lp
     \tilde{\cl}^{L} ( \partial VV ) (y^{n}_{\cdot})
    \rp_{su}\Big|
    \leq C^{7}_{ V,L}  P^{L}_{s}    
 \omega(s,u)^{1/p}  . 
  \label{eqn.dtnbd}
\end{align}

\end{lemma}
\begin{proof} Recall that $\cl^{L}$ is defined in \eref{eqn.ln}.  For two functions $f$, $g$: $[0,T]\to \RR$ and for the operator $\delta$ defined by \eqref{eq:def-delta}, it is easily seen that 
\begin{eqnarray}\label{eqn.deltafg}
\delta(fg)_{su} = \delta f_{su} \, g_{u}+f_{s} \, \delta g_{su}. 
\end{eqnarray}
   Invoking repeatedly this relation and consistently replacing the terms $g_{u}$ above by $g_{s}$ we end up with the relation 
\begin{align}\label{eqn.dln}
 & \delta\lp
     \cl^{L} f (y^{n}_{\cdot})
    \rp_{su} 
  =
    J^{1}_{su}
  + J^{2}_{su}, 
\end{align}
where   the terms $J^{1}_{su}$ and $J^{2}_{su}$ are defined by 
\begin{align}
J^{1}_{su} =&   \sum_{i\in \ca_{L}}  c_{L,i}
    \Big(
\langle
     \delta\lp    \partial^{\ell_{1}} f (y^{n}_{\cdot})
  \rp_{su}, 
   \Xi^{\ell_{2}}_{s} \otimes\cdots\otimes\Xi^{\ell_{\al_{i}}}_{s}  
   \rangle
  \nonumber
   \\
& \qquad\qquad  +
   \sum_{r=2}^{\al_{i}}
  \langle
       \partial^{\ell_{1}} f (y^{n}_{s}) , 
    \Xi^{\ell_{2}}_{s} 
\otimes   \cdots
 \otimes  \delta  
    \Xi^{\ell_{r}}_{su}\otimes\cdots\otimes\Xi^{\ell_{\al_{i}}}_{s} 
    \rangle 
 \Big) , 
 \label{eqn.j1}
\\
J^{2}_{su} =&  \sum_{i\in \ca_{L}}  c_{L,i}
    \Big\{
   \Big\langle  \delta\lp    \partial^{\ell_{1}} f (y^{n}_{\cdot})
  \rp_{su},
 \delta \lp
   \Xi^{\ell_{2}}_{\cdot}  \otimes\cdots\otimes\Xi^{\ell_{\al_{i}}}_{\cdot}   
 \rp_{su}
 \Big\rangle
\nonumber
  \\
  & \qquad\qquad+
   \sum_{r=2}^{\al_{i}}
  \Big\langle     \partial^{\ell_{1}} f (y^{n}_{s})
 ,
    \Xi^{\ell_{2}}_{s} 
  \otimes \cdots
 \otimes  \delta  
    \Xi^{\ell_{r}}_{su}
  \otimes  \delta  \lp
    \Xi^{\ell_{r+1}}_{\cdot}
   \otimes \cdots\otimes\Xi^{\ell_{\al_{i}}}_{\cdot}   
 \rp_{su} 
 \Big\rangle
 \Big\}. 
 \label{eqn.j2}
\end{align}
In order to bound $J^{1}_{su}$ above we  apply    Lemma \ref{lem.pn2} to the quantities $\Xi^{r}$,     and   Lemma  \ref{lem.dV} to $   \delta\lp    \partial^{\ell_{1}} f (y^{n}_{\cdot})
  \rp_{su} $ in \eref{eqn.j1}.   We get 
  \begin{align}\label{eqn.j1bd}
|J^{1}_{su}| \leq \sum_{i\in \ca_{L}} c_{L,i} \left( C^{0}_{f}K^{0}_{1}  P^{L}_{s}+\sum_{r=2}^{\al_{i}}C^{0}_{f}K^{\ell_{r}}_{1} P^{L}_{s} \right)  \omega(s,u)^{1/p} . 
\end{align}

 We can  bound $J^{2}_{su}$, in a similar way. As before we apply  Lemma \ref{lem.pn2} and   Lemma  \ref{lem.dV} respectively to $\Xi^{r}_{s}$     and    $   \delta\lp    \partial^{\ell_{1}} f (y^{n}_{\cdot})
  \rp_{su} $, and then  apply Lemma \ref{lem.dXi} to  the quantity $\delta  (
    \Xi^{\ell_{r+1}}_{\cdot}
   \otimes \cdots\otimes\Xi^{\ell_{\al_{i}}}_{\cdot}   
 )_{su} $. Taking into acount  the assumption    $\omega(s,u)^{1/p}<1/2$ in \eref{eqn.cdxi}, we obtain  
\begin{align}\label{eqn.j2bd}  
|J^{2}_{su}|
 \leq 
 2 \sum_{i\in \ca_{L}}  c_{L,i}  
  C^{0}_{f}   \Big( K_{1}^{0} C_{L}^{4}  P^{L}_{s}  +  \sum_{r=2}^{\al_{i}}    K_{1}^{\ell_{r}}  C_{L}^{4}  P^{L}_{s}    \Big)
   \cdot  
   \omega(s,u)^{1/p} 
    . 
\end{align}

Combining   the estimates \eref{eqn.j1bd} and \eref{eqn.j2bd} in  \eref{eqn.dln} and  recalling  the definition of $C^{5}_{f,L}$ in~\eref{eqn.c5},   relation   \eref{eqn.dlnbd} is now easily obtained. 
\end{proof}

We end   this subsection with a result on the remainder  of $\cl^{L} V(y^{n}_{s})$ considered as a controlled process:  
\begin{lemma}\label{eqn.dlv}
Let $\omega $ be a control function on $\ll0,T\rr$. 
Suppose that 
\begin{align}\label{eqn.xi2bd}
|\delta \Xi^{L}_{su} - \cl^{L} V(y^{n}_{s})    \delta x_{su}
 -
      \cl^{L-1} V(y^{n}_{s})   \delta b_{su}  | \leq K^{L}_{2}  P^{L}_{s} \omega(s,u)^{2/p}, \qquad L=0,1,\dots, N 
\end{align} 
 for some  $(s,u)\in \cs_{2}(\ll 0, T \rr)$, where $P^{L}_{s}$ is the quantity given in Definition \ref{def.p}. Suppose that \eref{eqn.cdxi} and  \eref{eqn.cddy} holds for the same $(s,u) $. 
Then we have the following relation for $L\geq -1$: 
\begin{align*}
&
\left|
 \delta \lp
 \cl^{L} V(y^{n}_{\cdot})
 \rp_{su}
   -
      \bar{\cl}^{L } \lp \partial V  V \rp (y^{n}_{s})    \delta x_{su} 
      -   \tilde{\cl}^{L  } \lp \partial V  V \rp (y^{n}_{s})    \delta b_{su}
   \right| 
   \leq 
   C^{8}_{V, L}  P^{L}_{s}    
 \omega(s,u)^{2/p}  , 
\end{align*}
where we define the constants  $\{ C^{8}_{V, L} , L\geq -1\}$ by 
\begin{align*}
 C^{8}_{V, L}  = 2\sum_{i\in \ca_{L}} c_{L,i}  
  C^{0}_{V} 
\lp
   K^{0}_{1}K^{0}_{1}+K^{0}_{2}  
+ 
    K_{0}^{1} C^{4}_{L}  
  +
    \sum_{r=2}^{\al_{i}}      (
    K^{\ell_{r}}_{2}
+
 K_{1}^{\ell_{r}}  C^{4}_{L}
 )      
\rp, 
\\
 C^{8}_{V, 0}  = C^{0}_{V} 
  (K^{0}_{1}K^{0}_{1} +K^{0}_{2} ), 
  \qquad 
  C^{8}_{V,-1}=0 . 
\end{align*}

\end{lemma}
\begin{proof}
Recall that   $\cl^{L}$, $\bar{\cl}^{L}$ and $\tilde{\cl}^{L}$ are introduced   
in \eref{eqn.ln}-\eref{eqn.ltn}. Similarly to the beginning of the proof of Lemma \ref{eqn.dlpvv}, we apply relation \eqref{eqn.deltafg} and replace the terms $g_{u}$ by $g_{s}$. This leads to a decomposition of the form  
\begin{align}\label{eqn.dlnvr}
&
 \delta \lp
 \cl^{L} V(y^{n}_{\cdot})
 \rp_{su}
   -
      \bar{\cl}^{L } \lp \partial V  V \rp (y^{n}_{s})    \delta x_{su} 
      -   \tilde{\cl}^{L  } \lp \partial V  V \rp (y^{n}_{s})    \delta b_{su}
  =   
  \sum_{\ell=1}^{4} J_{su}^{\ell}  ,   
\end{align}
where $J^{1}_{su} $  and $J^{2}_{su} $ are defined in \eref{eqn.j1}-\eref{eqn.j2}, and we also introduce the increments
\begin{align*}
J^{3}_{su} =&-
  \sum_{ i \in  \ca_{L}}   
 c_{L,i}
 \langle
(\partial^{\ell_{1}+1 } V  V)(y^{n}_{s}) \delta x_{su}   ,  
 \Xi^{\ell_{2}}_{s} \otimes  \cdots\otimes   \Xi^{\ell_{\al_{i}}}_{s} 
 \rangle 
 \\
 J^{4}_{su}  = & - \sum_{ i \in  \ca_{L}} 
\sum_{r=2}^{\al_{i}} 
 c_{L,i}    
 \langle
\partial^{\ell_{1}  } V(y^{n}_{s})     ,  
 \Xi_{s}^{\ell_{2}}  \otimes  \cdots\otimes 
\lp
 \cl^{\ell_{r}} V(y^{n}_{s}) \delta x_{su} +\cl^{
 \ell_{r}-1}V(y^{n}_{s}) \delta b_{su} 
\rp\otimes
 \cdots\otimes  \Xi_{s}^{\ell_{\al_{i}}} \rangle .   
\end{align*}
   Notice that one can combine $J^{1},J^{3}$ and $J^{4}$ into
   \begin{align*}
 &
 J^{1}_{su} +J^{3}_{su} +J^{4}_{su}   
 \\
 &=  
  \sum_{ i \in  \ca_{L}}   
 c_{L,i}
 \left\langle
 \Big(
   \delta\lp    \partial^{\ell_{1}} V (y^{n}_{\cdot})
  \rp_{su} -
(\partial^{\ell_{1}+1 } V  V)(y^{n}_{s}) \delta x_{su} 
\Big)
  ,  
 \Xi^{\ell_{2}}_{s} \otimes  \cdots \otimes  \Xi^{\ell_{\al_{i}}}_{s}  
 \right\rangle
 \\
 &
 \qquad
 + \sum_{ i \in  \ca_{L}} 
\sum_{r=2}^{\al_{i}} 
 c_{L,i}    
\left\langle
\partial^{\ell_{1}  } V(y^{n}_{s})     ,  
 \Xi_{s}^{\ell_{2}}   \otimes \cdots\otimes 
\Big(
\delta \Xi^{\ell_{r}}_{su}
-
 \cl^{\ell_{r}} V(y^{n}_{s}) \delta x_{su} -\cl^{
 \ell_{r}-1}V(y^{n}_{s}) \delta b_{su} 
\Big)\otimes
 \cdots  \otimes\Xi_{s}^{\ell_{\al_{i}}} \right\rangle . 
\end{align*}
We are now in a position to apply Lemma \ref{lem.ddV} and condition \eref{eqn.xi2bd} in order to get: 
\begin{align*}
| J^{1}_{su} +J^{3}_{su} +J^{4}_{su} | \leq
&
 \sum_{i\in \ca_{L}} c_{L, i}  P^{L}_{s}    C^{0}_{V} 
  (K^{0}_{1}K^{0}_{1} +K^{0}_{2} )\omega (s,u)^{2/p}  
  \\
  &
  +\sum_{i\in \ca_{L}}\sum_{r=2}^{\al_{i}}c_{L,i} C^{0}_{V}        K^{\ell_{r}}_{2}  P^{L}_{s}  \omega(s,u)^{2/p}.  
\end{align*}
Combining this estimate and the relation \eref{eqn.j2bd}  in \eref{eqn.dlnvr}, and taking into account the definition of the constant $C^{8}_{ V,L} $  we obtain the desired relation.  
\end{proof}

 \subsection{Upper-bound estimate of the derivatives}\label{section.bound}
 
 In this subsection, we derive a uniform upper-bound estimate for the  Malliavin derivatives of $y^{n}$. 
 For a given threshold $\al>0$, 
  our estimates consist of three parts, which are estimates of the derivative over the steps of  (1) small size  ($ <\!\!< \al$);
 (2) medium size ($\approx \al$); (3) large size ($>\!\!> \al$). 
 
 We now specify our threshold parameter $\al$. Towards this aim, recall that   $K^{L}_{1}$ and $K^{L}_{2}$ for $L=0,1,\dots, N$, are introduced in \eref{eqn.k12} and $K_{\mu}$ is defined in \eqref{eqn.kmu}.  We also define: 
\begin{align}\label{eqn.k34}
 K_{3}^{L} = K_{\mu} K_{4}^{L}  \vee 1, \quad \text{where}
 \quad
 K_{4}^{L} =  (  C^{8}_{V, L} + C^{8}_{V, L-1} +4 C^{6}_{V,L}+4C^{7}_{V,L})  \vee 1. 
\end{align}
Then we shall resort to a positive  constant $\al$ such that:   
 \begin{align}\label{eqn.al}
\al^{1/p}=&\min
\{
  1/2,   1/{K_{2}^{L}}  ,  1/{K_{3}^{L} },    L=0,1,\dots, N\}. 
\end{align}
 Eventually we introduce some second chaos processes which play a prominent role in the analysis of Euler schemes (see \cite{LT}). Namely for $[s,t] \in\ll0,T\rr$ we set 
 \begin{eqnarray}\label{eqn.q}
q_{st}^{ij}= \sum_{s\leq t_{k}<t} \lp   x_{t_{k}t_{k+1}}^{2,ij} -\frac12 \Delta^{2H} \mathbf{1}_{\{i=j\}} \rp \, ,
\quad\text{and}\quad
q^{b, ij}_{st} = \sum_{s\leq t_{k}<t} \lp   b_{t_{k}t_{k+1}}^{2,ij} -\frac12 \Delta^{2H} \mathbf{1}_{\{i=j\}} \rp   .
\end{eqnarray}
For convenience we also introduce a specific notation for the cross integrals between the independent fractional Brownian motions $x$ and $b$. Namely we first introduce a Gaussian process $\text{w}$ which encompasses the coordinates of both the driving nose $x$ and the extra noise~$b$. Specifically we define
\begin{eqnarray}\label{eqn.w.def}
\delta \text{w}_{st} := (\delta \text{w}_{st}^{1}, \dots, \delta \text{w}_{st}^{2d}    ):=  (\delta x_{st}^{1}, \dots,\delta x_{st}^{d}, \delta b_{st}^{1},   \dots, \delta b_{st}^{ d}    ) . 
\end{eqnarray}
 Then 
 writing $\text{w}^{2}$ for the iterated integral of $\text{w}$ (see Definition \ref{def:rough-path}) 
 we set 
\begin{align}
\tilde{\text{w}}_{st} =&
  \lp
 \text{w}^{2, ij}_{st}
 ,\, i=d+1, \dots, 2d ,\, j=1,\dots, d \rp   
 \label{eqn.wt}
\\
\tilde{q}_{st} =
&
 \lp
\sum_{s\leq t_{k}<t} \text{w}^{2, ij}_{t_{k}t_{k+1}}
, \, i=d+1, \dots, 2d ,\, j=1,\dots, d \rp   \, ,
\label{eqn.qt}
\end{align}
for $s,t\in \cs_{2}([0,T])$.  
With this notation in hand, we now state our bound on derivatives of the Euler scheme. 
 
 \begin{theorem}\label{thm.bd}
 Let $y$ and $y^{n}$ be the solution of the SDE \eqref{e.sde} and the corresponding Euler scheme~\eref{eqn.euler}, respectively. Let $\Xi$ be given in Definition \ref{def.xi}. 
 Suppose that $V\in C^{(L+2)\vee 3}_{b}$ for some integer  $L\geq 0$. Let $p>1/H$. Let  $\emph{\text{w}}=(x,b)$ be defined in \eqref{eqn.w.def} and let  $\bfw:=S_{2}(\emph{\text{w}})$ be the rough path lifted from $\emph{\text{w}}$. 
    We introduce a control $\omega$ by  
  \begin{align}\label{eqn.control}
  \omega (s,t) = \| \bfw\|_{p\text{-var}; \ll s,t\rr}^{p} + \|q\|_{p/2\text{-var};  \ll s,t\rr }^{p/2} + \|q^{b}\|_{p/2\text{-var};  \ll s,t\rr }^{p/2},  
    \qquad (s,t)\in  \cs_{2}(\ll 0, T\rr),  
\end{align}
where $q$ is defined in \eref{eqn.q}. 
  Denote $s_{0}=0$. Then
given $ s_{j}$, we define $s_{j+1}$ recursively as
\begin{equation}\label{eqn.sj}
s_{j+1}=
\begin{cases}
s_{j}+\Delta \, ,
&\textnormal{if } \omega(s_{j}, s_{j}+\Delta) > \al \\
\max  \{u \in \ll 0,T \rr :\, u>s_{j} \text{ and }   \omega(s_{j}, u)\leq \al\} \, ,
&\textnormal{if } \omega(s_{j}, s_{j}+\Delta) \le \al
\end{cases}
\end{equation}
Next we split the set of $s_{j}$'s as   
 \begin{align}
&S_{0} = \{s_{j}:  \al/2\leq \omega(s_{j}, s_{j+1})\leq \al\};
\quad
 S_{1}=  \{s_{j}:    \omega(s_{j}, s_{j+1}  )< \al/2\};
 \label{eqn.s01}
\\
&S_{2}=  \{s_{j}:    \omega(s_{j}, s_{j+1} )> \al\}. 
\label{eqn.cs2}
\end{align}
 Then we have:
 \begin{enumerate}[wide, labelwidth=!, labelindent=0pt, label= \emph{(\alph*)}]
\setlength\itemsep{.05in}
 \item 
 The following relation holds for all   $(s,t) \in \cs_{2}(\ll 0,T \rr)$:
\begin{align}\label{eqn.z.pbd}
 \|\Xi^{L}\|_{p\tvr, \ll s,t\rr} \leq K   \cdot \omega(s,t)^{1/p} |S_{0}\cup S_{1}\cup S_{2}| \cdot (\cm_{0}\cdot \cm_{1}\cdot \cm_{2})^{L}\,,
\end{align}
where we have set 
\begin{align}\label{eqn.ms}
&\cm_{0} =  
\prod_{s_{j}\in S_{0} } \lp K \omega(s_{j},s_{j+1})^{1/p}   +1\rp,
\qquad \cm_{1} =  
\prod_{s_{j}\in   S_{1}} \lp K \omega(s_{j},s_{j+1})^{1/p}   +1\rp,
\nonumber
 \\
& \cm_{2} =   
 \prod_{s_{j}\in S_{2} }
  \lp K  |  \delta{\emph{\text{w}}}_{s_{j}s_{j+1}}| +K\Delta^{2H} +1\rp 
 ,
\end{align}
and $K$ is a constant independent of $n$. 

\item
For $(s,t)\in \cs_{2}(\ll s_{j},s_{j+1} \rr)$ such that  $s_{j}\in S_{0}\cup S_{1}$ we have 
\begin{eqnarray*}
|\delta \Xi^{L}_{st} - \cl^{L} V(y^{n}_{s})    \delta x_{st}
 -
      \cl^{L-1} V(y^{n}_{s})   \delta b_{st}  | \leq K    \omega(s,t)^{2/p} \cdot P^{L}_{s}. 
\end{eqnarray*}
\end{enumerate}
 \end{theorem}
 
\begin{remark}\label{remark.hnorm}
The reader might argue that the right-hand side of~\eqref{eqn.z.pbd} still depends on $n$. However,
in our companion paper \cite{LLT1} we will show that this right-hand side is uniformly integrable in $n$. Thus 
$\Xi^{L}_{t}$ is also uniformly integrable in $n$. 
 \end{remark}

Before proving Theorem \ref{thm.bd}, let us state a corollary giving the actual bound on the Malliavin derivatives of $y^{n}$. Recall that $b$ is an independent copy of the fBm $x$ as given in Definition \ref{def.xi} and we  denote $\hat{D}$ the Malliavin derivative operator for $b$. We  denote 
$\hat{\mE}$ and $\hat{\DD}^{L,p}$ the   expectation and the Soblev space corresponding to $b$, respectively.

\begin{cor}\label{cor.dy.bd}
Under the same conditions as for Theorem \ref{thm.bd} and recalling our notation from Section \ref{subsection.d}, we have 
   \begin{eqnarray*}
\sup_{n\in\NN}\| {\bar{D}}^{L}  y^{n}_{t}\|_{\bar{\ch}^{\otimes L}} \leq K   \cdot \omega(0,T)^{1/p} |S_{0}\cup S_{1}\cup S_{2}| \cdot (\cm_{0}\cdot \cm_{1}\cdot \cm_{2})^{L}.
\end{eqnarray*}
\end{cor}

\begin{proof}
Because $\Xi^{L}_{t}$ as a functional of $b$ is in a finite chaos, we have $\|\Xi^{L}_{t}\|_{\hat{\mathbb{D}}^{L,p}}\leq C (\hat{\mE}|\Xi^{L}_{t}|^{p})^{1/p} $. 
Our claim is thus  an easy consequence of \eqref{eqn.z.pbd} combined with Lemma   \ref{lem.dhd}. 
\end{proof}

 \begin{proof}[Proof of Theorem \ref{thm.bd}]
 The proof is divided into several steps.

\noindent \textit{Step 1. Representation of remainders.} 
  Recall the definition of $\Xi^{L}$ in \eqref{eqn.xin}, that is
   \begin{eqnarray}\label{eqn.dxil}
\delta    \Xi^{L}_{ t_{k}t_{k+1} } &=&  
   \cl^{L}   V(y^{n}_{t_{k}})    \delta x_{t_{k}t_{k+1}}
 +
    \cl^{L-1} V(y^{n}_{t_{k}})  \delta b_{t_{k}t_{k+1}} 
\\
&&+ 
\frac12  \sum_{j=1}^{d}  \bar{\cl}^{L}
\lp \partial V_{j}\cdot V_{j}\rp (y^{n}_{t_{k}}) 
  \Delta^{2H}
+ 
\frac 12  \sum_{j=1}^{d}  \tilde{\cl}^{L-1}
\lp \partial V_{j}\cdot V_{j}\rp (y^{n}_{t_{k}}) 
  \Delta^{2H} .
  \nonumber
\end{eqnarray}
Next  observe that owing to our definition \eqref{eqn.q}, \eqref{eqn.wt} and \eqref{eqn.qt}  we have
\begin{eqnarray*}
\frac12 \Delta^{2H} \id_{d}= x^{2}_{t_{k}t_{k+1}}-q_{t_{k}t_{k+1}}, 
\quad\quad
\frac12 \Delta^{2H} \id_{d} = b^{2}_{t_{k}t_{k+1}}-q^{b}_{t_{k}t_{k+1}} 
\quad\text{and}\quad
\tilde{\text{w}}^{2}_{t_{k}t_{k+1}} - \tilde{q}_{t_{k}t_{k+1}}=0.  
\end{eqnarray*}
 Hence recalling our notation \eqref{eqn.wt}-\eqref{eqn.qt}, one can recast \eqref{eqn.dxil} as
  \begin{eqnarray}\label{eqn.dxil.new}
\delta    \Xi^{L}_{ t_{k}t_{k+1} } &=&  
   \cl^{L} V(y^{n}_{t_{k}})    \delta x_{t_{k}t_{k+1}}
 +
      \cl^{L-1} V(y^{n}_{t_{k}})   \delta b_{t_{k}t_{k+1}} 
 + 
  \bar{\cl}^{L} \lp \partial V  V \rp (y^{n}_{t_{k}})    (x^{2}_{t_{k}t_{k+1}}-q_{t_{k}t_{k+1}})
     \nonumber
\\
&& + 
     \tilde{\cl}^{L-1} \lp \partial V  V \rp (y^{n}_{t_{k}})   ( b^{2}_{t_{k}t_{k+1}} -  {q}^{b}_{t_{k}t_{k+1}} )
+ 
   \tilde{\cl}^{L} \lp \partial V  V \rp (y^{n}_{t_{k}})    (\tilde{\text{w}}^{2}_{t_{k}t_{k+1}}-\tilde{q}_{t_{k}t_{k+1}})
     \nonumber
 \\
&&
  + 
     \bar{\cl}^{L-1} \lp \partial V  V \rp (y^{n}_{t_{k}})   ( \tilde{\text{w}}^{2}_{t_{k}t_{k+1}} - \tilde{q}_{t_{k}t_{k+1}} )^{T}
 . 
\end{eqnarray}
 This suggests to define  the following  remainder process for $L=0,1,\dots, N$ and $s,t\in \ll0,T\rr$:  
 \begin{align} \label{eqn.r}
 R^{L}_{st}
  =& -\delta\Xi^{L}_{ st } +    \cl^{L} V(y^{n}_{s})    \delta x_{st}
 +
      \cl^{L-1} V(y^{n}_{s})   \delta b_{st} 
      \nonumber
\\
&
\qquad+ 
  \bar{\cl}^{L} \lp \partial V  V \rp (y^{n}_{s})    (x^{2}_{st}-q_{st})
 + 
     \tilde{\cl}^{L-1} \lp \partial V  V \rp (y^{n}_{s})   ( b^{2}_{st} -  {q}^{b}_{st} )
     \nonumber
 \\
&
\qquad+ 
   \tilde{\cl}^{L} \lp \partial V  V \rp (y^{n}_{s})    (\tilde{\text{w}}^{2}_{st}-\tilde{q}_{st})
 + 
     \bar{\cl}^{L-1} \lp \partial V  V \rp (y^{n}_{s})   ( \tilde{\text{w}}^{2}_{st} - \tilde{q}_{st} )^{T}
 . 
\end{align}
Notice that as a straightforward consequence of \eqref{eqn.dxil.new} we have   $R_{t_{k}t_{k+1}}^{L}=0$ for all $k$. 

Our aim is now to prove that $R^{L}$ is indeed a small remainder by applying Lemma \ref{lem2.4}. Since $R_{t_{k}t_{k+1}}^{L}=0$ it remains to analyze the increment $\delta R$ as defined in \eqref{eq:def-delta}. Now starting from \eqref{eqn.r}, an elementary computation yields:  
\begin{align*}
\delta R^{L}_{sut} = E_{sut}^{1}+\cdots+E_{sut}^{5},
\end{align*}
where we define
\begin{align*}
E_{sut}^{1}= & -  
\delta \lp
 \cl^{L} V(y^{n}_{\cdot})
 \rp_{su}
    \delta x_{ut}
 \\
  E_{sut}^{2}=&
  -  
\delta \lp
 \cl^{L-1} V(y^{n}_{\cdot})
 \rp_{su}
    \delta b_{ut}
 \\
 E_{sut}^{3}=&
      \bar{\cl}^{L } \lp \partial V  V \rp (y^{n}_{s})    \delta x_{su}\otimes \delta x_{ut}
  +
    \tilde{\cl}^{L} \lp \partial V  V \rp (y^{n}_{s})    \delta b_{su}\otimes \delta x_{ut} \\
 E_{sut}^{4}
 = &
     \tilde{\cl}^{L-1 } \lp \partial V  V \rp (y^{n}_{s})    \delta b_{su}\otimes \delta b_{ut}
 + 
   \bar{\cl}^{L-1} \lp \partial V  V \rp (y^{n}_{s})   \delta x_{su}\otimes \delta b_{ut}
    \\ 
    E_{sut}^{5} 
    =& 
 - 
  \delta\lp
    \bar{\cl}^{L} \lp \partial V  V \rp (y^{n}_{\cdot})
    \rp_{su}
        (x^{2}_{ut}-q_{ut})
  - 
  \delta\lp
    \tilde{\cl}^{L-1} \lp \partial V  V \rp (y^{n}_{\cdot})
    \rp_{su}
        (b^{2}_{ut}-q^{b}_{ut})
 \\
 &
 -\delta\lp
  \tilde{\cl}^{L} \lp \partial V  V \rp (y^{n}_{\cdot})   
  \rp_{su}
   (\tilde{\text{w}}^{2}_{ut}-\tilde{q}_{ut})
  - 
  \delta \lp
     \bar{\cl}^{L-1} \lp \partial V  V \rp (y^{n}_{\cdot})   
     \rp_{su}
     ( \tilde{\text{w}}^{2}_{ut} - \tilde{q}_{ut} )^{T}
   . 
\end{align*}
In the following steps we estimate the terms $E^{1}$, \dots, $E^{5}$ differently on small and large steps.

\noindent \textit{Step 2. Estimate over small and medium size steps.}  
We consider the intervals $[s_{j}, s_{j+1}]$ such that $s_{j}\in S_{0}\cup S_{1}$, with $S_{0}$ and $S_{1}$ defined by \eqref{eqn.s01}. 
In the following, we show   by induction  that for $s,u, t\in [s_{j}, s_{j+1}]$  the following   inequalities for 
        $  L=0,1,\dots, N$ are satisfied:
\begin{align}
\label{eqn.indxi}
&|\delta \Xi^{L}_{st}| \leq K^{L}_{1}  P^{L}_{s} \omega(s,t)^{1/p}, 
\
|\delta \Xi^{L}_{st} - \cl^{L} V(y^{n}_{s})    \delta x_{st}
 -
      \cl^{L-1} V(y^{n}_{s})   \delta b_{st}  | \leq K^{L}_{2}  P^{L}_{s} \omega(s,t)^{2/p}, 
      \\ 
   \label{eqn.indr}   
&      | R_{st}^{L}| \leq  K^{L}_{3}  P^{L}_{s} \omega (s,t)^{3/p} \, ,
      \qquad\quad
        | \delta R_{sut}^{L}| \leq  K^{L}_{4} P^{L}_{s} \omega (s,t)^{3/p}
, 
\end{align}
where we recall that $K_{1}^{L}$, \dots,  $K_{4}^{L}$ are defined in \eqref{eqn.k12} and \eqref{eqn.k34}, and where $P^{L}_{s}$ is introduced in Definition \ref{def.p}. 
Namely, 
suppose that \eref{eqn.indxi}-\eref{eqn.indr} hold for $s,u,t \in [s_{j}, v]$. In the following, we are going to show that   \eref{eqn.indxi}-\eref{eqn.indr} also holds on $[s_{j}, v +\Delta]$.   

To this aim, it is enough to estimate $\delta R^{L}_{sut}$ for $s,u\in [s_{j}, v]$ and $t\in [v,v+\Delta]$. For such a tuple $(s,u,t)$, we apply   Lemma \ref{eqn.dlpvv} to $E_{5}$  and Lemma  \ref{eqn.dlv} to $(E_{1}+E_{3})$ and $(E_{2}+E_{4})$.  We obtain  
\begin{align}\label{eqn.drbd}
|\delta R^{L}_{sut}| \leq
&
 (C^{8}_{V,L} P^{L}_{s}+C^{8}_{V,L-1} P^{L-1}_{s})     
 \omega(s,t)^{3/p} 
 +
 4(C^{6}_{L,V}+C^{7}_{L,V}) P^{L}_{s}    
 \omega(s,t)^{3/p} 
\nonumber
 \\
 \leq  &
 K^{L}_{4} P^{L}_{s} \omega(s,t)^{3/p}
 .  
\end{align}
From \eqref{eqn.drbd}, one can thus 
  complete the proof of the second inequality  in \eref{eqn.indr} by induction. The first inequality  in \eref{eqn.indr} is then obtained from the second one by a direct application of   Lemma \ref{lem2.4}.   

We now turn our attention to the proof of \eqref{eqn.indxi}. 
By applying relation \eref{eqn.indr},  \eref{eqn.lbnbd} and~\eref{eqn.ltnbd} to \eref{eqn.r} and taking into account the condition $\omega (s,t)^{1/p}\leq 1/K^{L}_{3}$, we obtain 
\begin{multline}\label{eqn.dxbd}
|\delta \Xi^{L}_{st} - \cl^{L} V(y^{n}_{s})    \delta x_{st}
 -
      \cl^{L-1} V(y^{n}_{s})   \delta b_{st}  |  \\
       \leq 
     ( C^{2}_{L,V}+C^{2}_{L-1,V}+ C^{3}_{L,V}+C^{3}_{L-1,V} +  1 )
    P^{L}_{s} 
     \omega (s,t)^{2/p}
  .  
\end{multline}
This concludes the proof of the second relation in \eref{eqn.indxi}. In order to show the first relation we apply \eref{eqn.lnbd} to  \eref{eqn.dxbd} and take into account the assumption that $ \omega (s,t)^{1/p}\leq 1/K^{L}_{2} $. We get 
\begin{align*}
|\delta \Xi^{L}_{st}  
    |   \leq  
 (C^{1}_{L,V}+C^{1}_{L-1,V}+1) P^{L}_{s} 
     \omega (s,t)^{1/p}
    . 
\end{align*}
This completes the proof of \eref{eqn.indxi}-\eref{eqn.indr} for $s,u,t \in [s_{j}, v +\Delta]$, under the hypothesis $s_{j}\in \cs_{0}\cup \cs_{1}$. 
Our induction procedure is thus achieved.

\noindent \textit{Step 3. Estimate over large size steps.}\quad
For large size steps, we will use a cruder estimate.   
Namely, when  $s_{j}$ sits in the set $S_{2}$ defined by \eqref{eqn.cs2},  we have $s_{j+1}=s_{j}+\Delta$.  It   follows  from  equation~\eqref{eqn.dxil}   that
\begin{align}\label{eqn.dzdelta}
|\delta \Xi^{L}_{s_{j}  s_{j+1}}|
=|\delta \Xi^{L}_{s_{j}, \, s_{j}+\Delta}|
\leq ( C^{1}_{L,V}+C^{1}_{L-1,V}+C^{2}_{L,V}+C^{3}_{L-1,V} ) P_{s_{j}}^{L}\lp   |\text{w}^{1}_{s_{j}  s_{j+1}}|+  \Delta^{2H} \rp  ,
\end{align}
where $\text{w}^{1}_{st}=\delta \text{w}_{st}$, 
and thus with Definition \ref{def.p} in mind we simply get 
\begin{align}\label{eqn.zdelta}
|  \Xi^{L}_{  s_{j+1}}|\leq P^{L}_{s_{j}} \lp K  |\text{w}^{1}_{s_{j}  s_{j+1}}|+K  \Delta^{2H} +1\rp .
\end{align}

\noindent \textit{Step 4. Conclusion.}\quad We first  derive the uniform upper-bound for $ \Xi^{L}$. That is, for  $t\in \ll s_{j},s_{j+1}\rr$ we write
\begin{align}\label{eqn.zs0}
|\Xi^{L}_{t}| \leq |\delta \Xi^{L}_{s_{j}t}|+|\Xi^{L}_{s_{j}}| . 
\end{align}
 Hence in the case   $s_{j}\in S_{0}\cup S_{1}$,   applying \eref{eqn.indxi} to \eref{eqn.zs0} we have 
 \begin{align*}
|\Xi^{L}_{t}| \leq  P^{L}_{s_{j}} \cdot \lp 1+K^{L}_{1} \omega(s_{j}, t)^{1/p} \rp. 
\end{align*}
Moreover,  in the case that $s_{j}\in S_{2} $, relation \eref{eqn.zdelta} implies that
  \begin{align*}
|\Xi^{L}_{t}| \leq  P^{L}_{s_{j}}  ( K  |\text{w}^{1}_{s_{j}  s_{j+1}}|+K  \Delta^{2H} +1 ). 
\end{align*}
 Iterating the above two estimates, we end up with 
 \begin{align}\label{eqn.zbd}
|\Xi^{L}_{t}| \leq K\cdot (\cm_{0}\cdot \cm_{1}\cdot \cm_{2})^{L} \, ,
\qquad \text{for all }
t\in \ll 0,T\rr,  
\end{align}
which is our desired uniform bound.

 We turn to the estimate of the increments of $\Xi^{L}$. 
 We first write
 \begin{align*}
|\delta \Xi^{L}_{st}| \leq \sum_{s\leq s_{j}<t} |\delta \Xi^{L}_{s\vee s_{j}, t\wedge s_{j+1}}|. 
\end{align*}
We apply  the increment inequalities \eref{eqn.indxi} for small sized steps  and \eref{eqn.dzdelta} for large sized steps. We also take   into account the uniform  estimate \eref{eqn.zbd}.  We then obtain
\begin{align}\label{eqn.dxig}
|\delta \Xi^{L}_{st}| \leq 
K\lp  \sum_{s_{j}\in S_{0}\cup S_{1}} \omega(s_{j}, s_{j+1})^{1/p} +
K\sum_{s_{j}\in S_{2}} \lp
 |\text{w}^{1}_{s_{j}  s_{j+1}}|+K  \Delta^{2H}
 +1
 \rp
  \rp  (\cm_{0}\cdot \cm_{1}\cdot \cm_{2})^{L}.
\end{align}
In the right-hand side of \eqref{eqn.dxig}, bounding each $\omega(s_{i}, s_{i+1})$ by $\omega (s,t)$ for every $s_{j}  \in S_{0}\cup S_{1} $ we get 
\begin{align}\label{eqn.sumo}
\sum_{s_{j}\in S_{0}\cup S_{1}} \omega(s_{j}, s_{j+1})^{1/p}  \leq \omega (s,t)^{1/p} |S_{0}\cup S_{1}|. 
\end{align}
In addition, recall that the control $\omega$ is defined by \eqref{eqn.control}, which includes the term $\|\omega\|_{p\text{-var}}^{p}$. Hence if $s_{j}\in S_{2}$ (that is $\omega(s_{j},s_{j+1})>\al$) and $\Delta$ is small enough, we have 
 \begin{align}\label{eqn.sumo2}
K\sum_{s_{j}\in S_{2}}  |\text{w}^{1}_{s_{j}  s_{j+1}}|+K  \Delta^{2H}
 \leq K \omega(s,t)^{1/p}  |S_{2}|. 
\end{align} 
Plugging \eqref{eqn.sumo} and \eqref{eqn.sumo2} into \eqref{eqn.dxig}, we obtain 
\begin{eqnarray}\label{eqn.dxibound}
|\delta \Xi_{st}^{L}| \leq K   \cdot \omega(s,t)^{1/p} |S_{0}\cup S_{1}\cup S_{2}| \cdot (\cm_{0}\cdot \cm_{1}\cdot \cm_{2})^{L} \, ,
\end{eqnarray}
for all $(s,t)\in \cs_{2}(\ll0,T\rr)$. The upper-bound   \eref{eqn.z.pbd} for $| \delta \Xi^{L}_{st}|$ is exactly \eqref{eqn.dxibound}.    
\end{proof}

 \subsection{Point-wise upper-bound estimate}\label{section.bound2}
 
 Theorem \ref{thm.bd} and Corollary \ref{cor.dy.bd} provide estimates in $p$-variation for the Malliavin derivatives of the Euler scheme $y^{n}$.
 In this subsection we use   similar  arguments in order    to derive a  pointwise estimate for the Malliavin derivatives of the form $Dy_{t}$ (recall that those $\ch$-valued derivatives have been defined in Section~\ref{subsection.d}).  
Notice that we have stated and proved the theorem below for the first two derivatives of $y^{n}$. However, the extension of this result to higher order derivatives is just a matter of cumbersome notation.     
 
 \begin{theorem}\label{thm.xirr}
Let the notations  in Theorem \ref{thm.bd} prevail.  Suppose that $V\in C_{b}^{4}$. 
  Take $r,r'\geq 0$ and let $k_{0},k_{0}'\in \NN$ be such  that $r\in (t_{k_{0}}, t_{k_{0}+1}] $ and $r'\in (t_{k_{0}'}, t_{k_{0}'+1}] $.
We define a Malliavin derivative vector $\xi^{n} $ as
\begin{equation}\label{d1}
\xi^{n} _{t}= (D_{r}y^{n}_{t} , D_{r}D_{r'}y^{n}_{t} ) := (\xi_{t}^{n,1}, \xi_{t}^{n,2})
\end{equation}
 Then  $\xi^{n,L} $, $L=1,2$ satisfies the iterative equation \eqref{eqn.xin} with $c_{L,i}=1$ and $\tilde{c}_{L,i}=0$. Furthermore,  we have the estimate 
\begin{align}\label{eqn.z.pbd1}
\|\xi^{n,1}\|_{p\tvr, \ll s,t\rr} +\| \xi^{n,2}\|_{p\tvr, \ll s,t\rr} \leq K \cdot \omega(s,t)^{1/p} |S_{0}\cup S_{1}\cup S_{2}| \cdot (\cm_{0}\cdot \cm_{1}\cdot \cm_{2})^{L} \, ,
\end{align}
for all $(s,t) \in \cs_{2}\ll 0,T \rr$, 
where $S_{i}$,
 $\cm_{i}$ are respectively  defined for $i=0,1,2$ in \eqref{eqn.s01}-\eqref{eqn.cs2} and~\eqref{eqn.ms}. Moreover, for
  $(s,t)\in \cs_{2}(\ll s_{j},s_{j+1} \rr)$ such that  $s_{j}\in S_{0}\cup S_{1}$ and $L=1,2$ we have 
\begin{eqnarray}\label{eqn.dxiestimate2}
|\delta \xi^{n,L}_{st} - \cl^{L} V(y^{n}_{s})    \delta x_{st}
   | \leq K    \omega(s,t)^{2/p} \cdot P^{L}_{s}.
\end{eqnarray}
In both estimates   
\eqref{eqn.z.pbd1} and \eqref{eqn.dxiestimate2}, $K$ is a constant independent of $r$, $r'$, $j$ and $n$. 
 
 \end{theorem}
 
 \begin{proof}
  Recall that $y^{n}$ is defined  in \eref{eqn.euler2}. Let us first  derive the iterative equation for the derivatives of $y^{n}$. 
Recall that $r\in (t_{k_{0}}, t_{k_{0}+1}]$.  
We can divide the differentiation of $\delta y^{n}_{t_{k}t_{k+1}}$ in three cases. 
\begin{enumerate}[wide, labelwidth=!, labelindent=0pt, label=\emph{(\roman*)}]
\setlength\itemsep{.1in}

\item
If $k_{0}>k$, then $r>t_{k+1}$. Therefore, since $y_{t_{k}} \in \cf_{t_{k}}$ we get
$D_{r} [ \delta y^{n}_{ t_{k} t_{k+1}}]=0$.

\item
If $k_{0}=k$,  
 then $t_{k}<r\leq t_{k+1}$. Moreover it is readily checked from equation~\eqref{e.df.finite} that $D_{r} [\delta x_{t_{k}t_{k+1}}]=\mathbf{1}_{[t_{k}, t_{k+1}]} (r)$. Hence we have $\mathbf{1}_{(t_{k},t_{k+1}]} (r)=1$ almost everywhere  and differentiating \eref{eqn.euler2} on both sides we obtain 
\begin{align*}
   D_{r} [\delta y^{n }_{t_{k}t_{k+1}}] =  \sum_{j=1}^{d}V_{j}(y^{n}_{t_{k}}) \mathbf{1}_{[t_{k}, t_{k+1}]} (r) =\sum_{j=1}^{d} V_{j}(y^{n}_{t_{k_{0}}}) \equiv a_{1}  .
\end{align*}

\item
If $k_{0}<k$, then we can  differentiate both sides of \eref{eqn.euler2}.  We obtain  the equation:
 \begin{eqnarray*}
\delta D_{r}   y^{n}_{t_{k}t_{k+1}}  =   \langle \partial V(y^{n}_{t_{k}}) , D_{r} y^{n}_{t_{k} } \rangle \delta x_{t_{k}t_{k+1}} + \frac12  \sum_{j=1}^{d}  \langle \partial (\partial V_{j} V_{j}) (y^{n}_{t_{k}}),   D_{r} y^{n}_{t_{k}} 
\rangle \Delta^{2H}  . 
\end{eqnarray*}
\end{enumerate}
Gathering item  (i), (ii) and (iii) above, and recalling that we have set $a_{1}=\sum_{j=1}^{d} V_{j}(y^{n}_{t_{k_{0}}})$, we get the following expression for $t\in \ll0,T\rr$, $r\leq t_{k}$ and $\xi^{n,1}_{t}=D_{r}y^{n}_{t}$:  
\begin{eqnarray}\label{eqn.xi1}
\delta\xi_{t_{k}t_{k+1}}^{n,1} = \langle \partial V(y^{n}_{t_{k}}) , \xi_{t_{k}}^{n,1} \rangle \delta x_{t_{k}t_{k+1}} + \frac12   \sum_{j=1}^{d}  \langle \partial (\partial V_{j} V_{j}) (y^{n}_{t_{k}}),    \xi_{t_{k}}^{n,1} 
\rangle \Delta^{2H} .  
\end{eqnarray}
Also notice that we have obtained   $ \xi_{t }^{n,1}=0$ if $r>t$.  In particular,  it is clear that $\xi^{n,1}$ satisfies the iteration equation in Definition \ref{def.xi} with   $c_{L,i}=1$ and $\tilde{c}_{L,i}=0$, with $L=1$ and initial time $t_{0}=t_{k_{0}}$. Precisely,    
 we have 
\begin{eqnarray}\label{eqn.xi1i}
\delta\xi^{n,1}_{t_{k}t_{k+1}} =   \cl^{1}_{\xi}V(y^{n}_{t_{k}}) \delta x_{t_{k}t_{k+1}} +  \frac12  \sum_{j=1}^{d}  \bar{\cl}^{1}_{\xi}  (\partial V_{j} V_{j}) (y^{n}_{t_{k}})  \Delta^{2H}  . 
\end{eqnarray}
Therefore, a direct application of   Theorem \ref{thm.bd} yields   the estimate \eqref{eqn.z.pbd1}   for $\xi^{n,1}$. 

 In a similar way we can show      that  $\xi^{n,2}_{t} = D_{r}D_{r'}y^{n}_{t}$ satisfies the iterative equation in Definition~\ref{def.xi}  with   $c_{L,i}=1$ and $\tilde{c}_{L,i}=0$, with $L=2$ and initial time $t_{0}=t_{k_{0}}\vee t_{k_{0}'}$. Indeed, a straightforward computation shows that  
 \begin{eqnarray*}
D_{r'}\xi^{n,1}_{t} = \xi^{n,2}_{t} ,
\qquad
D_{r'}\cl^{1}_{\xi}V(y^{n}_{t_{k}})=\cl^{2}_{\xi}V(y^{n}_{t_{k}}), 
\qquad
D_{r'} \bar{\cl}^{1}_{\xi}  (\partial V_{j} V_{j}) (y^{n}_{t_{k}})= \bar{\cl}^{2}_{\xi}  (\partial V_{j} V_{j}) (y^{n}_{t_{k}}). 
\end{eqnarray*}
Let   $r\in [t_{k_{0}}, t_{k_{0}+1})$ and $r'\in [t_{k_{0}'}, t_{k_{0}'+1})$. Then, by differentiating both sides of \eqref{eqn.xi1i} by $D_{r'}$ and taking into acount the above three relations  we get for all $t\geq t_{0}$ 
 \begin{eqnarray*}
\xi^{n,2}_{t} = a_{2} + \sum_{r\vee r'\leq t_{k}<t} \cl^{2}_{\xi}V(y^{n}_{t_{k}}) \delta x_{t_{k}t_{k+1}} +  \frac12 \sum_{r\vee r'\leq t_{k} <t} \sum_{j=1}^{d}  \bar{\cl}^{2}_{\xi}  (\partial V_{j} V_{j}) (y^{n}_{t_{k}})  \Delta^{2H}     ,
\end{eqnarray*}
where   $a_{2}$  is the initial value of the iterative  equation defined as follows 
 $$a_{2}=\mathbf{1}_{\{ k_{0}\geq k_{0}' \}}\cdot \sum_{j=1}^{d} \langle \partial V_{j}(y^{n}_{t_{k_{0}}}) , D_{r'}y^{n}_{t_{k_{0}}} \rangle  +\mathbf{1}_{\{ k_{0}'\geq k_{0}  \}}\cdot\sum_{j=1}^{d} \langle \partial 
V_{j}(y^{n}_{t_{k_{0}'}}) , D_{r}y^{n}_{t_{k_{0}'}} \rangle . $$ 
We conclude that  the estimate  \eqref{eqn.z.pbd1} also holds   for $\xi^{n,2}$. The proof is complete. 
\end{proof}

   \section*{Acknowledgments}
Yanghui Liu is supported by the PSC-CUNY Award 66385-00 54.


\end{document}